\documentclass[reqno]{amsart}

\usepackage{color}

\newtheorem{thm}{Theorem}[section]
\newtheorem{lem}[thm]{Lemma}
\newtheorem{cor}[thm]{Corollary}
\newtheorem{prop}[thm]{Proposition}
\newtheorem{rems}[thm]{Remarks}
\newtheorem{rem}[thm]{Remark}

\newtheorem{deff}[thm]{Definition}

\numberwithin{equation}{section}

\newcommand{\R}{\mathbb{R}}
\newcommand{\C}{\mathbb{C}}

\newcommand{\mS}{\mathbb{S}}

\newcommand{\N}{\mathbb{N}}
\newcommand{\A}{\mathbb{A}}
\newcommand{\B}{\mathbb{B}}
\newcommand{\T}{\mathbb{T}}

\newcommand{\X}{\mathbb{X}}

\newcommand{\I}{\mathbb{I}}

\newcommand{\ml}{\mathcal{L}}

\newcommand{\¦}{\Arrowvert}

\newcommand{\ve}{\varepsilon}
\newcommand{\rd}{\mathrm{d}}

\newcommand{\dom}{\mathrm{dom}}

\DeclareMathOperator*{\essinf}{ess-inf}

\newcommand{\dhr}{\mathrel{\lhook\joinrel\relbar\kern-.8ex\joinrel\lhook\joinrel\rightarrow}}

\begin{document}


\title[An Epidemic Model with Infection Age and Spatial Diffusion]{Well-Posedness and Stability Analysis of an Epidemic Model with Infection Age and Spatial Diffusion}

\author{Christoph Walker}
\email{walker@ifam.uni-hannover.de}
\address{Leibniz Universit\"at Hannover\\ Institut f\" ur Angewandte Mathematik \\ Welfengarten 1 \\ D--30167 Hannover\\ Germany}
\date{\today}

\begin{abstract}
A compartment epidemic model for infectious disease spreading is investigated, where movement of individuals is governed by spatial diffusion. The model includes infection age of the infected individuals and assumes a logistic growth of the susceptibles. Global   well-posedness of the equations within the class of nonnegative smooth solutions is shown. Moreover, spectral properties of the linearization around a steady state are derived. This yields the notion of linear stability which is used to determine stability properties of the disease-free and the endemic steady state.
\end{abstract}



\maketitle

\section{Introduction}

We consider a compartment epidemic model for infectious disease spreading. The total population is divided into susceptible and infected individuals which move in space by diffusion. Infectious individuals are structured by infective age keeping track of the time elapsed since an individual first acquires the disease. A logistic growth is assumed for the susceptibles.

Let $S(t,x)$  and $I(t,a,x)$ be the densities of susceptible and infected individuals, respectively, at time $t\ge 0$, position $x\in\Omega$, and infection age $a\in (0,a_m)$, where $\Omega\subset \R^n$ with {  $n\in\N^*$} is a bounded, smooth domain, and $a_m\in [0,\infty)$ is the maximal invective age. 
The population of susceptible individuals is assumed to obey a logistic growth with intrinsic growth rate $\kappa_1>0$ and carrying capacity $\kappa_2>0$. Susceptible individuals are infected at a rate $b(a,x)$ by  infected individuals of invective age $a$ and position $x$. Infectious individuals die naturally and disease-induced at a combined rate $\mu(a,x)\ge 0$. They may recover and enter directly the class of susceptibles at a rate $r(a,x)\ge 0$.
We shall thus focus on the equations
\begin{subequations}\label{E}
\begin{align}
\partial_t S(t,x)&=d_1\Delta S(t,x)+\kappa_1\left(1-\dfrac{1}{\kappa_2}S(t,x)\right) S(t,x) \notag\\
&\qquad -S(t,x)\int_0^{a_m} b(a,x)\,I(t,a,x)\,\rd a +\int_0^{a_m} r(a,x)\,I(t,a,x)\,\rd a\,, \label{E3}\\
D I(t,a,x)&=d(a)\Delta I(t,a,x)-m(a,x)I(t,a,x)-r(a,x)I(t,a,x)\,, \label{E1}
\\
I(t,0,x)&=S(t,x)\int_0^{a_m} b(a,x)\,I(t,a,x)\,\rd a\,, \label{E2}
\end{align}
for  $t>0$, $a\in (0,a_m)$, and $x\in \Omega$. The differentiation operator $D$ in \eqref{E1} is defined as
$$
D I(t,a,\cdot):=\lim_{h\to 0^+} \frac{1}{h}\big(I(t+h,a+h,\cdot)-I(t,a,\cdot)\big) 
$$
and is thus, if $I$ is continuously differentiable with respect to $t$ and $a$, given by
$$
DI(t,a,\cdot)=\partial_t I(t,a,\cdot)+\partial_a I(t,a,\cdot)\,.
$$
For notational simplicity we will take a susceptible diffusion rate $d_1=1$ and consider a diffusion coefficient $d=d(a)>0$ for infected individuals dependent only upon infection age (though a space dependence does not alter the subsequent results).
The equations are supplemented by the initial conditions
\begin{align}
S(0,x)=S_0(x)\,,\quad I(0,a,x)&= I_0(a,x)\,,\qquad (a,x)\in (0,a_m)\times\Omega\,, \label{E4}
\end{align}
and the boundary conditions
\begin{align}
(1-\delta)S(t,x)+\delta \partial_\nu S(t,x)= 0\,,\qquad (1-\delta)I(t,a,x)+\delta \partial_\nu I(t,a,x)&= 0 \label{E5}
\end{align}
\end{subequations}
for $(t,a,x)\in (0,\infty)\times(0,a_m)\times\partial\Omega$. Here, $\delta\in \{0,1\}$ is fixed so that $\delta=0$ corresponds to Dirichlet boundary conditions, while $\delta=1$ yields Neumann boundary conditions {  with $\partial_\nu I=\nabla I\cdot \nu$ denoting the derivative in normal direction $\nu$ on the boundary $\partial\Omega$} (we treat the two cases simultaneously).

Age-structured compartment  epidemic models  and age-structured population models in general have been investigated since many years \cite{ThiemeBook,WebbBook,WebbSpringer}. The particular case of
equations~\eqref{E} without spatial diffusion (and $r=0$) was studied in~\cite{CaoYanXu}. Therein, criteria for stability and instability of the disease-free and the endemic steady states were obtained in dependence on the corresponding basic reproduction number. Moreover, conditions for the occurrence of Hopf bifurcation were presented. 

The inclusion of spatial heterogeneity in age-structured populations leads to additional technical difficulties in the analysis. We refer to the monograph \cite{WebbSpringer} for a general treatment of and a comprehensive overview on such problems. Regarding SIS- and SIR-models there are various linear and nonlinear variants involving age and spatial structure with Laplace diffusion. The list includes the pioneering works \cite{WebbARMA80,WebbJMB82} on epidemic models including incubation periods, followed by  \cite{FitzgibbonParrotWebb95,FitzgibbonParrotWebb96,KuboLanglais94,LanglaisBusenberg97} and, more recently, \cite{ChekrounKuniya19,ChekrounKuniya20a,ChekrounKuniya20b,DiBlasio10,DucrotMagal09,DucrotMagal10,DucrotMagal11,Kim06,KuniyaOizumi15} though these references are non-exhaustive. The cited papers address various questions under different modeling hypotheses, for instance related to well-posedness of the equations, existence and stability properties of disease-free and endemic steady states, disease persistence, traveling wave solutions, or numerical simulations. We also refer to \cite{KangRuanJMB21} and the references therein for age-structured  epidemic models describing long-distance spreading of diseases by nonlocal diffusion. 

In this research we shall focus on the particular model~\eqref{E}. We first prove the existence and uniqueness of positive, global, smooth solutions by using a semigroup approach relying on results outlined in~\cite{WebbSpringer}. We then show that the linearization of these equations around a steady state yields a strongly continuous semigroup, and we use the spectrum of its generator for determining linear stability properties of steady states. Without further assumptions we prove stability or instability of the disease-free steady state in dependence on the reproduction number.
In a particular case of~\eqref{E} assuming spatially homogeneous rates and Neumann boundary conditions we improve the local stability of this steady state to global stability. Moreover,  we investigate the stability of the endemic steady state.

In the following Section~\ref{Sec2} we present our main results. Section~\ref{Sec3} is dedicated to the proof of the well-posedness of~\eqref{E}. The details on the linearized problem in the general case are given in Section~\ref{Sec4} and then applied in Section~\ref{Sec4B}. The application to a simplified version of equations~\eqref{E} with Neumann boundary conditions and spatially homogeneous rates are presented in Section~\ref{Sec5}. Some technical results are postponed to the Appendix~\ref{Appendix}.


\section{Main Results}\label{Sec2}


We first state the result on the well-posedness of~\eqref{E} and then investigate linearized stability  of steady states.


\subsection{Well-Posedness}


In order to present our existence result, we set $J:=[0,a_m]$ and take without loss of generality $d_1=1$. We assume that
\begin{subequations}\label{A}
\begin{align} 
\kappa_1\,,\kappa_2>0
\end{align}
and
\begin{align}\label{A1}
d\in C^\rho(J)\,,\qquad d(a)\ge \underline{d}>0\,,\quad a\in J\,,
\end{align}
Moreover,
\begin{align}\label{A2}
m\in C^\rho\big(J,L_\infty(\Omega)\big)\,,\quad r\in C^\rho\big(J,L_\infty(\Omega)\big)\cap L_\infty(J,C^1(\bar\Omega)\big)\,,\qquad m,r\ge 0\,,
\end{align}
for some $\rho>0$ and  
\begin{align}\label{A3}
b\in L_\infty\big(J,C^1(\bar\Omega)\big)\,, \quad b\ge 0\,,\quad b\not\equiv 0\,.
\end{align}
\end{subequations}
Here, $C^\rho$ stands for $\rho$-H\"older continuous functions. The regularity assumptions on the data are mainly imposed in order to derive smooth solutions. Denote $\R^+:=[0,\infty)$ and $\dot{\R}^+:=(0,\infty)$.\\

We shall prove the following result on the existence, uniqueness, and regularity of global, positive solutions to~\eqref{E}:

\begin{thm}\label{T1}
Assume~\eqref{A} and $p\in \big(\max\{3n/4,2\},\infty\big)$. Then, given initial values \mbox{$S_0\in L_p^+(\Omega)$} and  \mbox{$I_0\in L_1\big(J,L_p^+(\Omega)\big)$}, there is a unique positive global  solution $(S,I)$ to~\eqref{E} such that 
$$
S\in C\big(\R^+,L_p^+(\Omega)\big)\cap C^1\big(\dot{\R}^+,L_p(\Omega)\big)\cap C\big(\dot{\R}^+,W_p^2(\Omega)\big)
$$
is a strong solution to \eqref{E3}, while
$$
I\in C\big(\R^+,L_1(J,L_p^+(\Omega))\big) \cap C\big(\dot{\R}^+,L_1(J,W_{p}^{2}(\Omega))\big)
$$
satisfies \eqref{E1} in the sense that
$$
D I(t,a)=\big(d(a)\Delta -m(a,\cdot) -r(a,\cdot)\big)I(t,a,\cdot)\ \text{ in }\ L_p(\Omega)
$$
for $t>0$ and a.e. $a\in (0,a_m)$. In fact, the  solution map $(t,(S_0,I_0))\mapsto (S,I)(t)$ defines a global semiflow on $L_p^+(\Omega)\times L_1(J,L_p^+(\Omega))$.
\end{thm}

The proof of Theorem~\ref{T1} is based on a semigroup representation of solutions in the spirit of~\cite{WebbSpringer} and on Banach's fixed point theorem. It is performed in several steps in Section~\ref{Sec3}. In fact, it is not restricted to the particular nonlinearities in~\eqref{E} and may rather be a template for similar problems. That the  solution map defines a global semiflow  paves the way to consider qualitative aspects of the model.


\subsection{Linearization Around Steady States}


Assume \eqref{A} and fix an arbitrary steady state $(S_*,I_*)$ to~\eqref{E}, i.e. a time-independent solution. The regularizing effects of the Laplacian implies that we may assume without loss of generality the regularity
\begin{equation}\label{ss1}
\begin{split}
&S_*\in W_{p}^{2}(\Omega)\,, \quad S_*>0 \ \text{ in }\ \Omega\,,\\
&I_*\in L_1\big(J,W_{p}^{2}(\Omega)\big)\cap W_1^1\big(J,L_p(\Omega)\big)\,,\quad I_*\ge 0  \ \text{ in }\ J\times\Omega\,,
\end{split}
\end{equation}
for some $p>n$.
The linearization of~\eqref{E} around the steady state~$(S_*,I_*)$ is
\begin{subequations}\label{LIN1}
\begin{align}
\partial_t S(t,x)&=\Delta S(t,x)+\kappa_1S(t,x)- \dfrac{2\kappa_1 S_*(x)}{\kappa_2} S(t,x)-S(t,x)\int_0^{a_m} b(a,x)\,I_*(a,x)\,\rd a \nonumber\\
&\quad  -S_*(x)\int_0^{a_m} b(a,x)\,I(t,a,x)\,\rd a +\int_0^{a_m} r(a,x)\,I(t,a,x)\,\rd a \,, \label{E3al1}\\
D I(t,a,x)&=d(a)\Delta I(t,a,x)-\big(m(a,x)+r(a,x)\big) I(t,a,x)\,, \label{E1al1}
\\
I(t,0,x)&=S_*(x)\int_0^{a_m} b(a,x)\,I(t,a,x)\,\rd a +S(t,x)\int_0^{a_m} b(a,x)\,I_*(a,x)\,\rd a\,, \label{E2al1}
\end{align}
for $(t,a, x)\in \R^+\times (0,a_m)\times \Omega$, and subject to the initial conditions
\begin{align}
S(0,x)=S_0(x)\,,\quad I(0,a,x)&= I_0(a,x)\,,\qquad (a,x)\in (0,a_m)\times\Omega\,, \label{E4al1}
\end{align}
and boundary conditions
\begin{align}
(1-\delta)S(t,x)+\delta \partial_\nu S(t,x)= 0\,,\qquad (1-\delta)I(t,a,x)+\delta\partial_\nu I(t,a,x)&= 0 \label{E5al1}
\end{align}
\end{subequations}
for $(t,a,x)\in (0,\infty)\times(0,a_m)\times\partial\Omega$. We shall show that the solutions $(S,I)$ to the linearization~\eqref{LIN1} are given by a strongly continuous semigroup on $L_p(\Omega)\times L_1(J,L_p(\Omega))$ with compact resolvent:

\begin{thm}\label{T3a}
Suppose~\eqref{A}  and $p>(2\vee n)$. Let $(S_*,I_*)$ be a steady state to~\eqref{E} satisfying~\eqref{ss1}. Then, the solution $(S,I)$ to the linearized equation~\eqref{LIN1} is given as $$(S,I)(t)=\mS_*(t)(S_0,I_0)\,,\quad t\ge 0\,,$$ where $(\mS_*(t))_{t\ge 0}$ is a strongly continuous semigroup on $L_p(\Omega)\times L_1(J,L_p(\Omega))$. Its generator has a compact resolvent and thus a pure point spectrum without finite accumulation point.
\end{thm}

Theorem~\ref{T3a} is a consequence of Theorem~\ref{T0} and Corollary~\ref{C43} from Section~\ref{Sec4}. There, we also present more precise information on the semigroup and its generator. Note that the semigroup $(\mS_*(t))_{t\ge 0}$ lacks positivity and thus less information on the spectrum of its generator  is available in general. We refer to Remark~\ref{R1a} for further details. \\

{  

\subsection{Linear Stability}


Due to Theorem~\ref{T3a}, one may characterize stability properties of steady states based on the linearization of \eqref{E} around these steady states. That linearized stability indeed determines the (asymptotic) stability of steady states in certain nonlinear population models including age- and spatial structure has recently been shown in \cite{WalkerInstabiliy,WalkerZehetbauer}. We refrain, however, to prove this for problem~\eqref{E}.  

Herein, we shall just call a steady state $(S_*,I_*)$  {\it linearly stable} if the generator of the semigroup associated  with the linearization~\eqref{LIN1} of \eqref{E} (given in Theorem~\ref{T3a}) has a spectrum lying entirely in the half plane $\mathrm{Re}\, \lambda<0$ while we call the steady state {\it linearly unstable} if there is a spectral point in the  half plane $\mathrm{Re}\, \lambda>0$ (see Definition~\ref{DefStable} and Remark~\ref{R1a} below for more details in this regard). 

We assume for simplicity that
\begin{equation}\label{Aa}
r\equiv 0 \,.
\end{equation} 
We provide  a stability analysis in $L_p(\Omega)\times L_1(J,L_p(\Omega))$ of the trivial and the disease-free steady states. To state precise results we introduce  the principal eigenvalue  $\mu_0$ of the Laplacian $-\Delta$ on $\Omega$ subject to either Dirichlet boundary conditions (hence $\mu_0>0$) or Neumann boundary conditions (hence $\mu_0=0$).

\begin{thm}\label{TX}
Suppose~\eqref{A}, \eqref{Aa},  and let $p>(2\vee n)$.

\begin{itemize}
\item[{\bf (a)}] The trivial steady state $(S_*,I_*)=(0,0)$ to \eqref{E} is linearly unstable in $L_p(\Omega)\times L_1(J,L_p(\Omega))$ if $\kappa_1>\mu_0$ and linearly stable if $\kappa_1<\mu_0$.\vspace{1mm}

\item[{\bf (b)}] There is a disease-free steady state $(S_*,I_*)=(\tilde S_*,0)$ to \eqref{E} with a smooth function $\tilde S_*>0$ if and only if $\kappa_1>\mu_0$.  In this case, the  disease-free steady state $(S_*,I_*)=(\tilde S_*,0)$ is unique, and there is a number $\mathsf{R}_0>0$ such that it is linearly stable in \mbox{$L_p(\Omega)\times L_1(J,L_p(\Omega))$} if $\mathsf{R}_0<1$ and linearly unstable if~$\mathsf{R}_0>1$. \vspace{1mm}

\item[{\bf (c)}] Let $\kappa_1>\mu_0$. There is no endemic state $(S_*,I_*)$ to \eqref{E} with $S_*, I_*\ge 0$ and $I_*\not\equiv 0$ if $\mathsf{R}_0\le 1$.
\end{itemize}
\end{thm}

The proof of Theorem~\ref{TX} is presented in Section~\ref{Sec4B}.
The reproduction number  $\mathsf{R}_0>0$ is defined in \eqref{R0} and corresponds to the spectral radius of a compact irreducible operator (defined in~\eqref{Qs}) depending on $\tilde S_*$. It is open whether there is an endemic state $(S_*,I_*)$ with $S_*, I_*\ge 0$ and $I_*\not\equiv 0$  if~$\mathsf{R}_0>1$ (see Remark~\ref{R5} in this regard). However, the existence of an endemic steady state in case that~$\mathsf{R}_0>1$ is easily obtained when assuming Neumann boundary conditions and spatially homogeneous rates $m$ and $b$ as shown in the next subsection.
}


\subsection{Linear Stability in a Particular Model with Neumann Boundary Conditions}

\begin{subequations}\label{AA1}
We give a more detailed account of Theorem~\ref{TX} in the particular case of 
Neumann boundary conditions 
\begin{align}
\delta=1\,, 
\end{align}
so that $\mu_0=0$, and spatially homogeneous data 
\begin{align}\label{A1c1}
&d\in C(J)\,,\quad d(a)\ge \underline{d}>0\,,\quad a\in J\,,\\
&m\in C(J)\,,\quad m\ge 0\,, \qquad b\in L_\infty(J)\,,\quad  b> 0\,,\label{A2c1}
\end{align}
\end{subequations}
that is, data only depending on age $a$. 
In this particular situation, besides the trivial steady state $(S_*,I_*)=(0,0)$ and the disease-free steady state $(S_*,I_*)=(\kappa_2,0)$, there is also an \emph{endemic steady state} $(\bar S_*,\bar I_*)$ provided that  $\mathsf{R}_0>1$, where
\begin{align}\label{R00}
\mathsf{R}_0 :=\kappa_2\int_0^{a_m} b(a) \Pi(a)\, \rd a 
\end{align}
with
$$
\Pi(a):=\exp\left(-\int_0^a m(\sigma)\,\rd \sigma\right)\,,\quad a\in J\,.
$$
It is given as
$$
\bar S_*:=\frac{\kappa_2}{\mathsf{R}_0 }\,,\qquad \bar I_*(a):= \frac{1}{\mathsf{R}_0 } \kappa_1\kappa_2\left(1-\frac{1}{\mathsf{R}_0 }\right)\Pi(a)\,,\quad a\in J\,.
$$
Linear stability and instability  of these steady states is determined by the basic reproduction number $\mathsf{R}_0$:

\begin{thm}\label{T2}
Assume~\eqref{Aa}, \eqref{AA1}, let $p>(2\vee n)$, and let $\mathsf{R}_0>0$ be defined in~\eqref{R00}. \vspace{1mm}

\begin{itemize}
\item[{\bf (a)}] The trivial steady state $(S_*,I_*)=(0,0)$ is linearly unstable in $L_p(\Omega)\times L_1(J,L_p(\Omega))$.\vspace{1mm}

\item[{\bf (b)}] If $\mathsf{R}_0 <1$, then the disease-free steady state $(S_*,I_*)=(\kappa_2,0)$ is globally linearly stable in $L_p(\Omega)\times L_1(J,L_p(\Omega))$; that is, it is linearly stable and attracts any solution starting from positive initial values. If $\mathsf{R}_0 >1$, then $(S_*,I_*)=(\kappa_2,0)$ is linearly unstable in $L_p(\Omega)\times L_1(J,L_p(\Omega))$. \vspace{1mm}

\item[{\bf (c)}] For $1<\mathsf{R}_0 <3$, the endemic steady state $(\bar S_*,\bar I_*)$ is linearly stable in $L_p(\Omega)\times L_1(J,L_p(\Omega))$.
\end{itemize}
\end{thm}

Part (a) and the local stability statements of part (b) of Theorem~\ref{T2} have been observed already in Theorem~\ref{TX}. The proofs of the remaining statements  of Theorem~\ref{T2} are given in Section~\ref{Sec5}. In fact, when $\mathsf{R}_0 <1$ we prove  that
$$
\lim_{t\to\infty} \big(S(t),I(t)\big)=(\kappa_2,0)\ \text{ in }\ L_p(\Omega)\times L_1\big(J, C(\bar\Omega)\big)
$$
for any solution $(S,I)$ to \eqref{E}  corresponding to positive nontrivial initial values $(S_0,I_0)$ so that there is no further steady state in this case (in accordance with Theorem~\ref{TX}~(c)). Clearly, one expects $(\bar S_*,\bar I_*)$ to be linearly stable whenever $\mathsf{R}_0 >1$.



\section{Well-Posedness: Proof of Theorem~\ref{T1}}\label{Sec3}

We prove Theorem~\ref{T1} in several steps. After introducing some notation we derive the existence of a local solution then establish further properties.

\subsection{Preliminaries $\&$ Notation}

{  For two Banach spaces $E$ and $F$ we write $\ml(E,F)$ for the Banach space of bounded linear operators from $E$ to $F$, and we set $\ml(E):=\ml(E,E)$. Similarly, $\mathcal{K}(E,F)$ and $\mathcal{K}(E)$ stand for compact linear operators.}

For fixed $\delta\in\{0,1\}$ and $p\in (1,\infty)$, we set
$$
\mathcal{B}u:=u  \ \text{ on } \ \partial\Omega \ \text{ if } \ \delta=0\,,\qquad \mathcal{B}u:=\partial_\nu u\ \text{ on } \ \partial\Omega \ \text{ if } \ \delta=1\,,
$$
and introduce the scale of Banach spaces
\begin{equation}\label{interpol}
 W_{p,\mathcal{B}}^{2\theta}(\Omega):=\left\{\begin{array}{ll} \{v\in W_{p}^{2\theta}(\Omega)\,;\; \mathcal{B} w=0 \text{ on } \partial\Omega\}\,, & \delta +\frac{1}{p}<2\theta\le 2\,,\\[3pt]
	 W_{p}^{2\theta}(\Omega)\,, & 0\le 2\theta<\delta +\frac{1}{p}\,.\end{array} \right.
\end{equation}
By $\Delta_\mathcal{B} $ we denote the Laplacian defined on $W_{p,\mathcal{B}}^2(\Omega)$.  Moreover, for fixed $a\in J$, also the operator
\begin{align}\label{Aaaa}
A(a):=d(a)\Delta_\mathcal{B} -m(a,\cdot)-r(a,\cdot)
\end{align}
has domain $W_{p,\mathcal{B}}^2(\Omega)$. Then $\Delta_\mathcal{B} $ and $A(a)$
are generators of positive analytic contraction semigroups on $L_p(\Omega)$ for each $p\in (1,\infty)$ \cite{AmannIsrael, Rothe}. In fact, since 
$$
A\in C^\rho\big(J,\ml(W_{p,\mathcal{B}}^2(\Omega),L_p(\Omega))\big)\,,
$$ 
it follows from \cite[II.Corollary~4.4.2]{LQPP} that $A$ generates a positive parabolic evolution operator 
$$
U_A(a,\sigma)\,,\quad 0\le\sigma\le a\le a_m\,,
$$ 
on $L_p(\Omega)$ in the sense of \cite[II.Section~2.1]{LQPP}. In particular,
$$
v(a):=U_{A}(a,\sigma)v^0\,,\quad a\in [\sigma,a_m]\,,
$$ 
is, for  given $\sigma\in [0,a_m)$ and $v^0\in L_p(\Omega)$, the unique solution
$$
v\in C\big([\sigma,a_m],L_p(\Omega)\big) \cap C^1\big((\sigma,a_m],L_p(\Omega)\big)\cap C\big((\sigma,a_m),W_{p,\mathcal{B}}^2(\Omega)\big)
$$
to the Cauchy problem
$$
\partial_a v(a)=A(a) v(a)\,,\quad a\in (\sigma,a_m)\,,\qquad v(\sigma)=v^0\,.
$$
The contraction properties
\begin{equation}
\begin{split}\label{EST2}
\|e^{t\Delta_\mathcal{B}}\|_{\ml(L_p(\Omega))}&\le 1\,,\quad t\ge 0\,,\\
\|U_A(a,\sigma)\|_{\ml(L_p(\Omega))} &\le  e^{-\int_\sigma^a\underline{m}(r)\,\rd r}\,\,,\quad 0\le\sigma\le a\in J\,,
\end{split}
\end{equation}
are valid, where 
$$
\underline{m}(a):=\essinf_{x\in \Omega} m(a,x)\ge 0\,,\quad a\in J\,.
$$
{  Recall the interpolation relations 
$$
\big(L_p(\Omega),W_{p,\mathcal{B}}^2(\Omega)\big)_{\theta,p}\doteq 
W_{p,\mathcal{B}}^{2\theta}(\Omega)\,,\quad 2\theta\in [0,2]\setminus\left\{1,\delta+\frac{1}{p}\right\}\,,
$$ 
with real interpolation functor $(\cdot,\cdot)_{\theta,p}$ and $$
\big[L_p(\Omega),W_{p,\mathcal{B}}^2(\Omega)\big]_{1/2}\doteq 
W_{p,\mathcal{B}}^{1}(\Omega)
$$ with complex interpolation functor $[\cdot,\cdot]_{1/2}$ (see~\cite[4.4.3/Theorem]{Triebel}).  From~\cite[II.~Lemma~5.1.3]{LQPP} and the embedding $W_{p,\mathcal{B}}^{2\theta}(\Omega)\hookrightarrow L_q(\Omega)$ for $\theta=\frac{n}{2}(\frac{1}{p}-\frac{1}{q})$ we then infer parabolic regularizing properties in the sense that, given $0\le\vartheta\le\theta\le 1$ with $2\vartheta, 2\theta\notin\{\delta+\frac{1}{p}\}$ and $1<p\le q\le\infty$, there are $\varpi\in \R$ and $M\ge 1$ such that
\begin{align}\label{EST1}
t^{\theta-\vartheta}\,\|e^{t\Delta_\mathcal{B}}\|_{\ml(W_{p,\mathcal{B}}^{2\vartheta}(\Omega),W_{p,\mathcal{B}}^{2\theta}(\Omega))}+ t^{\frac{n}{2}(\frac{1}{p}-\frac{1}{q})}\,\|e^{t\Delta_\mathcal{B}}\|_{\ml(L_p(\Omega),L_q(\Omega))} \le M \, e^{\varpi t}\,,\quad t>0\,,
\end{align}
and
\begin{align}\label{EST}
(a-\sigma)^{\theta-\vartheta}\, \|U_A(a,\sigma)\|_{\ml(W_{p,\mathcal{B}}^{2\vartheta}(\Omega),W_{p,\mathcal{B}}^{2\theta}(\Omega))}+ (a-\sigma)^{\frac{n}{2}(\frac{1}{p}-\frac{1}{q})}\, \|U_A(a,\sigma)\|_{\ml(L_p(\Omega),L_q(\Omega))}\le M \, e^{\varpi(a-\sigma)}\,.
\end{align}
}

Let $p\in \big(\max\{\frac{3n}{4},2\},\infty\big)$ and let $S_0\in L_p(\Omega)$ and $I_0\in L_1(J,L_p(\Omega))$ be fixed in the following.

\subsection{Existence of a Unique Maximal Solution}  
Given $S\in L_p(\Omega)$ and $I\in L_1(J,L_p(\Omega))$ we  use the abbreviations (dropping $x$-dependence for simplicity)
$$
B[S,I]:=S\int_0^{a_m} b(a)\,I(a)\,\rd a\,,\qquad R[I]:=\int_0^{a_m} r(a)\,I(a)\,\rd a\,,
$$
and
$$
 f[S,I]:=\kappa_1\left(1-\frac{S}{\kappa_2}\right) S-B[S,I]+R[I]\,,
$$
whereas, for time-dependent functions $S:[0,T]\to L_p(\Omega)$ and $I:[0,T]\to L_1(J,L_p(\Omega))$, it is convenient to abbreviate
$$
B[S,I](t):=B[S(t),I(t)]\,,\quad f[S,I](t):=f[S(t),I(t)]\,,\qquad t\in [0,T]\,.
$$
Then \eqref{E} can be written compactly as
\begin{subequations}\label{EE}
\begin{align}
\partial_t S(t)&=\Delta_\mathcal{B} S(t) +f[S,I](t)\,,\qquad t>0\,,\label{EE3}\\
D I(t,a)&=A(a)I(t,a)\,,\qquad t>0\,,\quad a\in J\,, \label{EE1}\\
I(t,0)&=B[S,I](t)\,,\qquad t>0 \,, \label{EE2}
\end{align}
subject to the initial conditions
\begin{align}
S(0)=S_0\,,\qquad I(0,a)&= I_0(a)\,,\quad a\in J\,.\label{EE4}
\end{align}
\end{subequations}
Solutions $S$ to \eqref{EE3} are of the form
\[
\mathcal{S}[S,I](t):= e^{t\Delta_\mathcal{B}}S_0+\displaystyle\int_0^t e^{(t-\tau)\Delta_\mathcal{B}}f[S,I](\tau)\,\rd \tau\,,\quad t>0\,,
\]
while integrating \eqref{EE1} subject to \eqref{EE2} formally along characteristics (and recalling the properties of the evolution operator $U_A$) yields a solution $I$ in the form
\[
\mathcal{I}[S,I](t,a):=\left\{\begin{array}{ll}
U_A(a,a-t) I_0(a-t)\,,& a>t\,, \ a\in J\,,\\[2pt]
U_A(a,0)B[S,I](t-a) \,,&a\le t\,,\ a\in J\,.
\end{array} \right.
\]
Given $T>0$ we introduce the Banach space
\[
\X_T:=C\big([0,T],L_p(\Omega)\times L_1(J,L_p(\Omega))\big)
\]
and define
\[
\mathcal{Y}[S,I](t):=\big(
\mathcal{S}[S,I](t) \,,\, \mathcal{I}[S,I](t,\cdot)\big)\,,\qquad t\in [0,T]\,,\quad (S,I)\in \X_T\,.
\]
Then, fixed points $(S,I)$ of $\mathcal{Y}$ correspond to solutions of \eqref{EE}. In order to prove that $\mathcal{Y}$ has a fixed point we first note:

\begin{lem}\label{L1w}
For $q=\frac{p}{2}$, the mappings
\begin{equation}\label{31}
B:L_p(\Omega)\times L_1\big(J,L_p(\Omega)\big)\to L_{q}(\Omega)\,,\qquad f:L_p(\Omega)\times L_1\big(J,L_p(\Omega)\big)\to L_{q}(\Omega)
\end{equation} 
are uniformly Lipschitz continuous on bounded sets. Moreover, if $2\theta>n/p$, then there is $\alpha>0$ such that
$$B:W_p^{2\theta}(\Omega)\times L_1\big(J,W_p^{2\theta}(\Omega)\big)\to W_p^{2\alpha}(\Omega)\,,\qquad f:W_p^{2\theta}(\Omega)\times L_1\big(J,W_p^{2\theta}(\Omega)\big)\to W_p^{2\alpha}(\Omega)$$ are uniformly Lipschitz continuous on bounded sets.
\end{lem}

\begin{proof}
The statements readily follow from the regularity assumptions~\eqref{A} and the fact that pointwise multiplications 
\begin{equation}\label{mult}
L_p(\Omega)\times L_p(\Omega)\to L_{p/2}(\Omega)\quad \text{ and }\quad W_p^{2\theta}(\Omega)\times W_p^{2\theta}(\Omega)\to W_p^{2\alpha}(\Omega)
\end{equation} 
are continuous for some $\alpha>0$, see \cite[Theorem~4.1]{AmannMultiplication}.
\end{proof}

\begin{prop}\label{P1}
Given $R>0$ there is $T=T(R)>0$ such that, if $$\|S_0\|_{L_p(\Omega)}+\|I_0\|_{L_1(J,L_p(\Omega))}<R\,,$$ then
\[
\mathcal{Y}:\bar\B_{\X_T}(0,R)\to \bar\B_{\X_T}(0,R)
\]
has a unique fixed point~$(S,I)$.
\end{prop}

\begin{proof}
Let $T\in (0,1)$ and $\|S_0\|_{L_p(\Omega)}+\|I_0\|_{L_1(J,L_p(\Omega))}<R$.  Considering $(S,I), (\tilde S,\tilde I)\in {\X_T}$ both with norm less than $R$, we have 
$$
f[S,I]\in C\big([0,T],L_{p/2}(\Omega)\big)
$$ 
by Lemma~\ref{L1w} so that we readily obtain that $\mathcal{S}[S,I]\in C\big([0,T],L_{p}(\Omega)\big)$ by~\eqref{EST1} since  $t\mapsto t^{-n/2p} $ is integrable on $(0,T)$ as $2p>n$. Moreover,
\begin{align}\label{q1}
\|\mathcal{S}[S,I](t)\|_{L_{p}(\Omega)}\le \|S_0\|_{L_p(\Omega)}+c(R)\, T^{1-n/2p} \,,\quad t\in [0,T]\,,
\end{align}
and
\begin{align}\label{q2}
\|\mathcal{S}[S,I](t)-\mathcal{S}[\tilde S, \tilde I](t)\|_{L_{p}(\Omega)}\le c(R)\, T^{1-n/2p}\,\|(S,I)-(\tilde S,\tilde I)\|_{\X_T}\,,\quad t\in [0,T]\,.
\end{align}
From \eqref{EST2}-\eqref{EST} and Lemma~\ref{L1w} we infer\footnote{Here and in the following, if   $t>a_m$, then integrals $\int_0^t\rd a$ equal $\int_0^{a_m}\rd a$ and integrals $\int_0^{a_m-t}\rd a$ vanish.}
\begin{align}
||\mathcal{I}[S,I](t)\|_{L_1(J,L_p(\Omega))}&\le \int_0^t\|U_A(a,0)\|_{\ml(L_{p/2}(\Omega),L_p(\Omega))}\,\|B[S,I](t-a)\|_{L_{p/2}(\Omega)}\,\rd a\nonumber\\
&\qquad + \int_t^{a_m}\|U_A(a,a-t)\|_{\ml(L_p(\Omega))}\,\|I_0(a-t)]\|_{L_p(\Omega)}\,\rd a\nonumber\\
&\le c(R) T^{1-n/2p}+\|I_0\|_{L_1(J,L_p(\Omega))}\label{q3}
\end{align}
for $t\in [0,T]$, and similarly
\begin{align}
\|\mathcal{I}[S,I](t)-&\mathcal{I}[\tilde S, \tilde I](t)\|_{L_{p}(\Omega)}\nonumber\\
&\le \int_0^t\|U_A(a,0)\|_{\ml(L_{p/2}(\Omega),L_p(\Omega))}\,\|B[S,I](t-a)-B[\tilde S,\tilde I](t-a)\|_{L_{p/2}(\Omega)}\,\rd a \nonumber \\
&\le c(R)\, T^{1-n/2p}\,\|(S,I)-(\tilde S,\tilde I)\|_{\X_T}\,. \label{q4}
\end{align}
To check continuity we use \eqref{EST2}-\eqref{EST} together with Lemma~\ref{L1w} and write, for $0\le t_2\le t_1\le T$,
\begin{align*}
||\mathcal{I}[S,I](t_1)&-\mathcal{I}[S,I](t_2)\|_{L_1(J,L_p(\Omega))}\\
&\le \int_0^{t_2}\|U_A(a,0)\|_{\ml(L_{p/2}(\Omega),L_p(\Omega))}\,\|B[S,I](t_1-a)-B[S,I](t_2-a)\|_{L_{p/2}(\Omega)}\,\rd a\\
&\quad +\int_{t_2}^{t_1}\|U_A(a,0)\|_{\ml(L_{p/2}(\Omega),L_p(\Omega))}\,\|B[S,I](t_1-a)\|_{L_{p/2}(\Omega)}\,\rd a\\
&\quad + \int_{t_2}^{t_1}\|U_A(a,a-t_2)\|_{\ml(L_{p}(\Omega))}\,\|I_0(a-t_2)\|_{L_p(\Omega)}\,\rd a\\
&\quad + \int_{t_1}^{a_m}\big\| \big(U_A(a,a-t_1)-U_A(a,a-t_2)\big) I_0(a-t_1)\|_{L_p(\Omega)}\,\rd a\\
&\quad + \int_{t_1}^{a_m} \| U_A(a,a-t_2)\|_{\ml(L_p(\Omega))}\,\| I_0(a-t_1)-I_0(a-t_2)\|_{L_p(\Omega)}\,\rd a\\
&\le Me^\varpi \int_0^{t_2} a^{-n/2p}\,\|B[S,I](t_1-a)-B[S,I](t_2-a)\|_{L_{p/2}(\Omega)}\,\rd a\\
&\quad +c(R) \int_{t_2}^{t_1}a^{-n/2p}\,\rd a + \int_{t_2}^{t_1}\|I_0(a-t_2)\|_{L_p(\Omega)}\,\rd a\\
&\quad + \int_{t_1}^{a_m}\big\| \big(U_A(a,a-t_1)-U_A(a,a-t_2)\big) I_0(a-t_1)\|_{L_p(\Omega)}\,\rd a\\
&\quad + \int_{t_1}^{a_m} \| I_0(a-t_1)-I_0(a-t_2)\|_{L_p(\Omega)}\,\rd a\,.
\end{align*}
Now, as $\vert t_1-t_2\vert \to 0$, the first integral on the right-hand side goes to zero since the function $B[S,I]\in C\big([0,T],L_{p/2}(\Omega)\big)$ is uniformly continuous while the second and the third integral vanish since $a\mapsto a^{-n/2p}$ respectively $I_0$ are integrable. To see that the fourth integral vanishes in the limit one may use the strong continuity of the evolution operator $U_A$ on $L_p(\Omega)$ \cite[Equation~II.~(2.1.2)]{LQPP} and Lebesgue's theorem. Finally, for the last integral one may use the strong continuity of the translations on $L_1(J,L_p(\Omega))$.
Consequently, $\mathcal{I}[S,I]\in C\big([0,T],L_1(J,L_p(\Omega))\big)$.

Summarizing, we have shown in \eqref{q1}-\eqref{q4} that, given $$
\|S_0\|_{L_p(\Omega)}+\|I_0\|_{L_1(J,L_p(\Omega))}<R\,,
$$ 
we can choose $T=T(R)\in (0,1)$ such that
\[
\mathcal{Y}:\bar\B_{\X_T}(0,R)\to \bar\B_{\X_T}(0,R)
\]
is a contraction, and the claim follows from Banach's fixed point theorem.
\end{proof}

Since $T=T(R)$ in the proof of Proposition~\ref{P1} depends only upon $$R>\|S_0\|_{L_p(\Omega)}+\|I_0\|_{L_1(J,L_p(\Omega))}\,,$$ it is standard to extend  $(S,I)$ to a maximal solution and to show that the solution map defines a semiflow:

\begin{cor}\label{C1a}
$(S,I)$ can be extended to a maximal interval $[0,T_m)$ such that 
$$(S,I)\in C\big([0,T_m),L_p(\Omega)\times L_1(J,L_p(\Omega))\big)$$ satisfies
\begin{equation}\label{S}
S(t)= e^{t\Delta_\mathcal{B}}S_0+\displaystyle\int_0^t e^{(t-\tau)\Delta_\mathcal{B}}f[S,I](\tau)\,\rd \tau\,,\quad t\in [0,T_m)\,,
\end{equation}
and
\begin{equation}\label{I}
I(t,a)=\left\{\begin{array}{ll}
U_A(a,a-t) I_0(a-t)\,,& a\ge t\,, \ \quad (a,t)\in J\times [0,T_m)\,,\\[2pt]
U_A(a,0)B[S,I](t-a) \,,&a< t\,,\ \quad (a,t)\in J\times [0,T_m)\,.
\end{array} \right.
\end{equation}
If $T_m<\infty$, then
\begin{equation}\label{T}
\lim_{t\nearrow T_m}\big(\|S(t)\|_{L_p(\Omega)}+\|I(t,\cdot)\|_{L_1(J,L_p(\Omega))}\big)=\infty\,.
\end{equation}
Moreover, the mapping $\big(t,(S_0,I_0)\big)\mapsto (S,I)(t)$ defines a semiflow on $L_p(\Omega)\times L_1(J,L_p(\Omega))$.
\end{cor}

\begin{rem}\label{R1}
It is worth noting that Corollary~\ref{C1a} remains valid for models that can be recast in the form~\eqref{EE} such that $f$ and $B$ satisfy~\eqref{31} with $\frac{n}{2}(\frac{1}{q}-\frac{1}{p})<1$ for some $1<q\le p<\infty$.
\end{rem}

\subsection{Regularity}\label{SReg}

We derive further regularity properties of the solution $(S,I)$ (it is for this step that we have imposed restrictive regularity assumptions on the data $b$ and $r$).

\begin{prop}\label{P2}
Let $2\vartheta\in [0,2]\setminus\{\delta+\frac{1}{p}\}$.
If $S_0\in W_{p,\mathcal{B}}^{2\vartheta}(\Omega)$, then
$$
S\in C^1\big((0, T_m),L_p(\Omega)\big)\cap C\big((0, T_m),W_{p,\mathcal{B}}^{2}(\Omega)\big)\cap C\big([0, T_m),W_{p,\mathcal{B}}^{2\vartheta}(\Omega)\big)
$$ 
is a strong solution to~\eqref{E3} while, if $I_0\in L_1(J,W_{p,\mathcal{B}}^{2\vartheta}(\Omega))$, then
\[
I\in C\big((0,T_m),L_1(J,W_{p,\mathcal{B}}^{2}(\Omega))\big)\cap C\big([0, T_m),L_1(J,W_{p,\mathcal{B}}^{2\vartheta}(\Omega)))\big)
\]
satisfies~\eqref{E1} in the sense that
$$
D I(t,a)=A(a)I(t,a)\ \text{ in }\ L_p(\Omega)
$$
for $t\in (0,T_m)$ and a.e. $a\in (0,a_m)$. 
\end{prop}

\begin{proof}
Since 
$$
\¦ e^{t\Delta_\mathcal{B}}\¦_{\ml (L_{p/2},W_{p,\mathcal{B}}^{2\theta}(\Omega))}\leq
c(T)t^{-n/2p-\theta}\ ,\quad 0< t\leq T\ ,
$$
and $f[S,I]\in C\big([0,T_m),L_{p/2}(\Omega)\big)$, it readily follows from~\eqref{S} that $S\in C\big((0,T_m),W_{p,\mathcal{B}}^{2\theta}(\Omega)\big)$ for $2\theta<2-n/p$ with $2\theta\notin\{\delta+1/p\}$. Similarly, as in the proof of Proposition~\ref{P1} (see also the proof of Lemma~\ref{APL1} in the Appendix) one derives from
\begin{align*}
 \|U_A(a,\sigma)\|_{\ml(L_{p/2}(\Omega),W_{p,\mathcal{B}}^{2\theta}(\Omega))}\le M (a-\sigma)^{-n/2p-\theta}\,,\quad 0\le\sigma\le a\in J\,,
\end{align*}
and $B[S,I]\in C\big([0,T_m),L_1(J,L_{p/2}(\Omega))\big)$ that $I\in C\big((0,T_m),L_1(J,W_{p,\mathcal{B}}^{2\theta}(\Omega))\big)$ for $2\theta<2-n/p$ with $2\theta\notin\{\delta+1/p\}$. 
Now, since $n<4p/3$, we find
$2\theta\in(n/2p,2-n/p)\setminus \{\delta+1/p\}$ so that, according to Lemma~\ref{L1w}, there is $\alpha>0$ such that
$$
f[S,I](\ve+\cdot)\in C\big([0, T_m-\ve),W_{p,\mathcal{B}}^{2\alpha}(\Omega)\big)
$$ for each $\ve>0$ small. Thus, we infer from \cite[II.Theorem~1.2.2]{LQPP} that
$$
S_\ve:=S(\ve+\cdot)\in C^1\big((0, T_m-\ve),L_p(\Omega)\big)\cap C\big((0, T_m-\ve),W_{p,\mathcal{B}}^{2}(\Omega)\big)
$$ 
is a strong solution to 
$$
\partial_t S_\ve=\Delta_\mathcal{B} S_\ve+f[S,I](\ve+\cdot)\,,\quad t\in (0, T_m-\ve)\,,\qquad S_\ve(0)=S(\ve)\in W_{p,\mathcal{B}}^{2\theta}(\Omega)\,.
$$
Letting then $\ve$ tend to zero we obtain that 
$$
S\in C^1((0, T_m),L_p)\cap C((0, T_m),W_{p,\mathcal{B}}^{2}(\Omega))
$$ 
is a strong solution to~\eqref{E3}. Moreover, if $S_0\in W_{p,\mathcal{B}}^{2\vartheta}(\Omega)$ for some $2\vartheta\in [0,2]\setminus\{\delta+1/p\}$, then
$$
S\in C\big([0, T_m),W_{p,\mathcal{B}}^{2\vartheta}(\Omega)\big)\,.
$$ 
Similarly, setting 
$$
I_\ve:=I(\ve+\cdot,\cdot)\,,\qquad I_{0,\ve}:=I_\ve(0,\cdot)=I(\ve,\cdot)\,,
$$
we deduce from \eqref{I} and the properties of evolution operators that, for $t\in [0,T_m-\ve)$ and $a\in J$, 
\begin{align*}
I_\ve(t,a)&=\left\{\begin{array}{ll}
U_A(a,a-t-\ve) I_0(a-t-\ve)\,,& a>t+\ve\,, \\[2pt]
U_A(a,0)B[S,I](\ve+t-a) \,,&a\le t+\ve\,, 
\end{array} \right.\\
&=\left\{\begin{array}{ll}
U_A(a,a-t)U_A(a-t,a-t-\ve) I_0(a-t-\ve)\,,& a>t+\ve\,,\\[2pt]
U_A(a,a-t)U_A(a-t,0)B[S,I](\ve+t-a) \,,&t<a\le  t+\ve\,, \\[2pt]
U_A(a,0)B[S,I](\ve+t-a) \,,&a\le t\,,
\end{array} \right.\\
&=\left\{\begin{array}{ll}
U_A(a,a-t) I_{0,\ve}(a-t)\,,& a>t\,, \\[2pt]
U_A(a,0)B[S_\ve,I_\ve](t-a) \,,&a\le t\,.
\end{array} \right.
\end{align*}
Now, since
$$
S_\ve\in  C\big([0, T_m-\ve),W_{p,\mathcal{B}}^{2}(\Omega)\big)\,,\quad I_\ve\in  C\big([0, T_m-\ve),L_1(J,W_{p,\mathcal{B}}^{2\theta}(\Omega))\big)\,,\quad  I_{0,\ve}\in L_1(J,W_{p,\mathcal{B}}^{2\theta}(\Omega))
$$
for $2\theta<2-n/p$, it follows from \eqref{EST} (see Lemma~\ref{APL1} in the Appendix) that
\[
I\in C\big((0,T_m),L_1(J,W_{p,\mathcal{B}}^{2}(\Omega))\big)\,.
\]
In addition, if $I_0\in L_1(J,W_{p,\mathcal{B}}^{2\vartheta}(\Omega))$ for some $2\vartheta\in [0,2]\setminus\{\delta+1/p\}$, then
\[
I\in C\big([0, T_m),L_1(J,W_{p,\mathcal{B}}^{2\vartheta}(\Omega)))\big)\,.
\]
Moreover, \eqref{I} and the differentiability properties of the evolution operator~$U_A$ stated in \cite[II.Equation~(2.1.6)]{LQPP} imply
$$
D I(t,a)=\lim_{h\to 0^+} \frac{1}{h}\big(I(t+h,a+h)-I(t,a)\big)=A(a)I(t,a)\ \text{ in }\ L_p(\Omega)
$$
for $t\in (0,T_m)$ and a.e. $a\in (0,a_m)$ (in fact, for every $a\in (0,a_m)$ if $I_0\in C\big((0,a_m), L_p(\Omega)\big)$ is continuous).
\end{proof}

Note that taking $2\vartheta=0$ in Proposition~\ref{P2} we obtain the regularity of the solution $(S,I)$ claimed in Theorem~\ref{T1}. 

\begin{rem}
Assuming additionally $b\in BC^1(J,C(\bar\Omega))$ and $I_0\in C^1(J,L_p(\Omega))$,  one can show analogously to \cite[Proposition~1]{WalkerDCDSA10} that the partial derivatives
$\partial_t I(t,a)$ and $\partial_a I(t,a)$
exist  and
\[
D I(t,a)= \partial_t I(t,a)+ \partial_a I(t,a)=A(a) I(t,a)
\]
in $L_p(\Omega)$ for $t\in (0,T_m)$ and $a\in (0,a_m)$.
\end{rem}

\subsection{Positivity}

Since the semigroup $\big(e^{t\Delta_\mathcal{B}}\big)_{t\ge 0}$ and the evolution operator $\big(U_A(a,\sigma)\big)_{0\le\sigma\le a\le a_m}$ are positive operators on $L_p(\Omega)$ (as well as on the spaces $W_{p,\mathcal{B}}^{2\theta}(\Omega)$) and since there is $\omega(R)>0$ such that
\[
B[S,I]\ge 0\,,\quad f[S,I]+\omega(R)S\ge 0
\]
provided that $S,I\ge 0$ with $\|(S,I)\|_{L_\infty(\Omega)\times L_1(J,L_\infty(\Omega))}\le R$ (see Subsection~\ref{SReg} for such local bounds), it is a standard iteration argument to derive that the solution $(S,I)$ from Corollary~\ref{C1a}  corresponding to non-negative initial values $S_0\in L_p^+(\Omega)$ and $I_0\in L_1(J,L_p^+(\Omega))$ satisfies \mbox{$S(t)\in L_p^+(\Omega)$} and \mbox{$I(t)\in L_1(J,L_p^+(\Omega))$} for $t\in [0,T_m)$.

\subsection{Global Existence}

Integrating~\eqref{E} yields for $t\in (0,T_m)$ the inequality (in fact,  equality for Neumann boundary conditions, see Section~\ref{APSec2} in the Appendix for a rigorous proof)
\begin{align}
\int_\Omega S(t,x)\,&\rd x+\int_0^{a_m}\int_\Omega I(t,a,x)\,\rd x\,\rd a\nonumber\\
\le &
\int_\Omega S_0(x)\,\rd x+\int_0^{a_m}\int_\Omega I_0(a,x)\,\rd x\,\rd a+\int_0^t\int_\Omega \kappa_1\left(1-\frac{S(\tau,x)}{\kappa_2}\right)S(\tau,x)\,\rd x\,\rd \tau\nonumber\\
& -\int_0^t\int_0^{a_m}\int_\Omega m(a,x)\,I(\tau,a,x)\,\rd x\,\rd a\,\rd\tau
-\int_{a_m-t}^{a_m}\int_\Omega I_0(a,x)\,\rd x\,\rd a\nonumber\\
&
+\int_0^t\int_{a_m-t+\tau}^{a_m}\int_\Omega \big(m(a,x)+r(a,x)\big)\,U_A(a,a-\tau)\,I_0(a-\tau,x)\,\rd x\,\rd a\,\rd\tau\,.\label{L1id}
\end{align}
Since
\begin{equation}\label{esS}
\kappa_1\left(1-\frac{S}{\kappa_2}\right)S\le \frac{\kappa_1\kappa_2}{4}\,,
\end{equation}
we thus deduce from the positivity of $(S,I)$ the $L_1$-estimate
\begin{equation}\begin{split} \label{L1est}
\|S(t)\|_{L_1(\Omega)}+\|I(t)\|_{L_1(J,L_1(\Omega))} \le \, & \|S_0\|_{L_1(\Omega)}+\|I_0\|_{L_1(J,L_1(\Omega))} \\
&+t\vert\Omega\vert \frac{\kappa_1\kappa_2}{4}+\|m+r||_{L_\infty(J,L_\infty(\Omega))} t\|I_0\|_{L_1(J,L_1(\Omega))}
\end{split}
\end{equation}
for $t\in (0,T_m)$. We shall then proceed with the following auxiliary result:

\begin{lem}\label{L2w}
{\bf (i)} Let $1\le q\le r\le\infty$ with $\frac{n}{2}(\frac{1}{q}-\frac{1}{r})<1$. If 
$$
\|I(t)\|_{L_1(J,L_q(\Omega))}\le c_0(T)\,,\quad t\in [0,T]\,,
$$
then
$$
\|S(t)\|_{L_r(\Omega)}\le \|S_0\|_{L_r(\Omega)}+c(T)\,,\quad t\in [0,T]\,.
$$

{\bf (ii)} Let $1\le r\le\infty$ with $\frac{n}{2r}<1$. If 
$$
\|S(t)\|_{L_r(\Omega)}\le c_0(T)\,,\quad t\in [0,T]\,,
$$
then
$$
\|I(t)\|_{L_1(J,L_p(\Omega))}\le  c(T)\big(1+\|I_0\|_{L_1(J,L_p(\Omega))}\big) \,,\quad t\in [0,T]\,.
$$
\end{lem}

\begin{proof}
{\bf (i)} By \eqref{S} we have
\begin{align*}
0\le S(t)\le\, & e^{t\Delta_\mathcal{B}}S_0+\displaystyle\int_0^t e^{(t-\tau)\Delta_\mathcal{B}}\kappa_1\left(1-\frac{S(\tau)}{\kappa_2}\right)S(\tau)\,\rd\tau\\
&
+\displaystyle\int_0^t e^{(t-\tau)\Delta_\mathcal{B}}\int_0^{a_m} r(a,\cdot)I(\tau,a)\,\rd a\,\rd\tau 
\end{align*}
and therefore, using \eqref{esS} and
$$
\|e^{(t-\tau)\Delta_\mathcal{B}}\|_{\ml(L_q(\Omega),L_r(\Omega))}\le c(T) (t-\tau)^{-\frac{n}{2}(\frac{1}{q}-\frac{1}{r})}\,,\quad 0\le \tau<t\le T\,,
$$
we deduce from $\|I(t)\|_{L_1(J,L_q(\Omega))}\le c_0(T)$ for  $t\in [0,T]$ that
\begin{align*}
\|S(t)\|_{L_r(\Omega)}&\le \|S_0\|_{L_r(\Omega)}+c \int_0^t\|e^{(t-\tau)\Delta_\mathcal{B}}\|_{\ml(L_\infty(\Omega))} \,\left\|\kappa_1\left(1-\frac{S(\tau)}{\kappa_2}\right)S(\tau) \right\|_{L_\infty(\Omega)} \,\rd\tau\\
&\qquad + \int_0^t \|e^{(t-\tau)\Delta_\mathcal{B}}\|_{\ml(L_q(\Omega),L_r(\Omega))}\,\|r\|_{L_\infty(J,L_\infty(\Omega))}\,\|I(\tau)\|_{L_1(J,L_q(\Omega))}\, \rd\tau \\
&\le \|S_0\|_{L_r(\Omega)}+c(T)
\end{align*}
for $t\in [0,T]$. This proves (i).\\

{\bf (ii)} Set $\frac{1}{q}:=\frac{1}{r}+\frac{1}{p}$ so that
\[
\| B[S,I]\|_{L_q(\Omega)}\le \|b\|_{L_\infty(J,L_\infty(\Omega))}\,\|S\|_{L_r(\Omega)}\, \|I\|_{L_1(J,L_p(\Omega))}\,. 
\]
Then we infer for $t\in [0,T]$ from \eqref{I}, \eqref{EST2}, and \eqref{EST}
that
\begin{align*}
\|I(t)\|_{L_1(J,L_p(\Omega))}&\le  \int_0^t \|U_A(a,0)\|_{\ml(L_q(\Omega),L_p(\Omega))}\, \| B[S,I](t-a)\|_{L_q(\Omega)}\,\rd a\\
&\qquad +\int_t^{a_m}\|U_A(a,a-t)\|_{\ml(L_p(\Omega))}\,\| I_0(a-t)\|_{L_p(\Omega)}\, \rd a\\
&\le c(T) \int_0^t (t-\sigma)^{-\frac{n}{2r}}\, \|I(\sigma)\|_{L_1(J,L_p(\Omega))}\,\rd  \sigma+\|I_0\|_{L_1(J,L_p(\Omega))} 
\end{align*}
with $\frac{n}{2r}<1$ whenever $\|S(t)\|_{L_r(\Omega)}\le c_0(T)$ for $ t\in [0,T]$. Hence, Gronwall's inequality implies
\begin{align*}
\|I(t)\|_{L_1(J,L_p(\Omega))}\le c(T)\big(1+\|I_0\|_{L_1(J,L_p(\Omega))}\big) \,,\quad t\in [0,T]\,,
\end{align*}
as claimed.
\end{proof}

Now, since
$$
\|I(t)\|_{L_1(J,L_1(\Omega))}\le c(T)\,,\quad t\in [0,T] \cap  [0,T_m)\,,
$$
by \eqref{L1est}, we deduce from Lemma~\ref{L2w}~(i) that
$$
\|S(t)\|_{L_r(\Omega))}\le c(T)\,,\quad t\in [0,T] \cap  [0,T_m)\,,
$$
for $n/2<r<n/(n-2)$ and hence
$$
\|I(t)\|_{L_1(J,L_p(\Omega))}\le c(T)\,,\quad t\in [0,T] \cap  [0,T_m)\,,
$$
due to  Lemma~\ref{L2w}~(ii). Taking $r=q=p$ in Lemma~\ref{L2w}~(i) yields now
$$
\|S(t)\|_{L_p(\Omega))}\le c(T)\,,\quad t\in [0,T] \cap  [0,T_m)\,.
$$
Consequently, $T_m=\infty$ according to~\eqref{T}. This completes the proof of Theorem~\ref{T1}.

\section{Linearized Stability of Steady States}\label{Sec4}

We  linearize~\eqref{E} around a steady state and then derive properties of the associated linear semigroup. This allows us to introduce the notion of linear stability.

Throughout this chapter we assume \eqref{A} and fix an arbitrary steady state $(S_*,I_*)$ to~\eqref{E} with regularity
\begin{equation}\label{ss}
\begin{split}
&S_*\in W_{p,\mathcal{B}}^{2}(\Omega)\,, \quad S_*>0 \ \text{ in }\ \Omega\,,\\
&I_*\in L_1(J,W_{p,\mathcal{B}}^{2}(\Omega))\cap W_1^1(J,L_p(\Omega))\,,\quad I_*\ge 0  \ \text{ in }\ J\times\Omega\,,
\end{split}
\end{equation}
for some $p>( 2\vee n)$. 


\subsection{Linearization Around Steady States}


Linearizing~\eqref{E} around the steady state~$(S_*,I_*)$ yields the problem
\begin{subequations}\label{LIN}
\begin{align}
\partial_t S(t,x)&=\Delta S(t,x)+\kappa_1S(t,x)- \dfrac{2\kappa_1 S_*(x)}{\kappa_2} S(t,x)-S(t,x)\int_0^{a_m} b(a,x)\,I_*(a,x)\,\rd a \nonumber\\
&\quad  -S_*(x)\int_0^{a_m} b(a,x)\,I(t,a,x)\,\rd a +\int_0^{a_m} r(a,x)\,I(t,a,x)\,\rd a \,, \label{E3al}\\
D I(t,a,x)&=d(a)\Delta I(t,a,x)-\big(m(a,x)+r(a,x)\big) I(t,a,x)\,, \label{E1al}
\\
I(t,0,x)&=S_*(x)\int_0^{a_m} b(a,x)\,I(t,a,x)\,\rd a +S(t,x)\int_0^{a_m} b(a,x)\,I_*(a,x)\,\rd a\,, \label{E2al}
\end{align}
for $(t,a,x)\in \R^+\times [0,a_m]\times \Omega$, and subject to the initial conditions
\begin{align}
S(0,x)=S_0(x)\,,\quad I(0,a,x)&= I_0(a,x)\,,\qquad (a,x)\in (0,a_m)\times\Omega\,, \label{E4al}
\end{align}
and boundary conditions
\begin{align}
\mathcal{B}S(t,x)= 0\,,\quad \mathcal{B}I(t,a,x)&= 0\,,\qquad (t,a,x)\in\R^+\times (0,a_m)\times\partial\Omega\,. \label{E5al}
\end{align}
\end{subequations}
Introducing 
\begin{subequations}\label{STAR}
\begin{align}
q_*:&=\int_0^{a_m} b(a,\cdot)\,I_*(a,\cdot)\,\rd a\in C^1(\bar\Omega)\,,\\
 P_*I&:=S_*\int_0^{a_m} b(a,\cdot)\,I(a,\cdot)\,\rd a\,,\qquad NI:=\int_0^{a_m} r(a,\cdot)\,I(a,\cdot)\,\rd a\,, \label{qnot}
\end{align}
and setting
\begin{align}
A_1^*&:=\Delta_\mathcal{B} +\kappa_1- \dfrac{2\kappa_1 S_*}{\kappa_2}-q_*\,,\label{A1not}\\
A(a)&:=d(a)\Delta_\mathcal{B} -m(a,\cdot)-r(a,\cdot)\,,\quad a\in J\,,\label{A1nott}
\end{align}
it follows
\begin{align}\label{Pstern}
P_*\,,\, N\in\ml\big(L_1(J,L_p(\Omega)),L_p(\Omega)\big)
\end{align} 
\end{subequations}
and $A_1^*$ with domain $W_{p,\mathcal{B}}^2(\Omega)$ generates a positive, compact, analytic semigroup $(e^{tA_1^*})_{t\ge 0}$ on $L_p(\Omega)$
while
the operator family $A(a)$ with domain $W_{p,\mathcal{B}}^2(\Omega)$ generates a positive parabolic evolution operator
$(U_{A}(a,\sigma))_{0\le \sigma\le a\le a_m}$
on $L_p(\Omega)$. With this notation we can recast the linearization~\eqref{LIN} as an equation in $L_p(\Omega)\times L_1(J,L_p(\Omega))$ of  the form
\begin{subequations}\label{C}
\begin{align}\label{C1}
\partial_t\left(\begin{matrix} S \\ I \end{matrix}\right)=\left(\begin{matrix} A_1^* & -P_* +N \\ 0& -\partial_a+A(a) \end{matrix}\right)\left(\begin{matrix} S \\ I \end{matrix}\right)\,, \quad t>0\,,\qquad \left(\begin{matrix} S \\ I \end{matrix}\right)(0)=\left(\begin{matrix} S_0 \\ I_0 \end{matrix}\right)\,,
\end{align}
subject to
\begin{align}\label{C2}
I(t,0)-P_*I(t,\cdot)=q_*S(t)\,,\quad t>0\,.
\end{align}
\end{subequations}
Following~\cite{WalkerIUMJ} we next show that the solutions to~\eqref{C}  are given by a strongly continuous semigroup on the phase space $L_p(\Omega)\times L_1(J,L_p(\Omega))$.

\subsection{The Semigroup Associated with the Linearization~\eqref{C}}

In order to investigate the properties of the semigroup generated by the linearization~\eqref{C} we write
\begin{align*} 
\left(\begin{matrix} A_1^* & -P_* +N \\ 0& -\partial_a+A(a) \end{matrix}\right)=\left(\begin{matrix} A_1^* & 0 \\ 0 & -\partial_a+A(a) \end{matrix}\right)+\left(\begin{matrix}0 & -P_* +N \\ 0& 0 \end{matrix}\right)
\end{align*}
and use a perturbation argument, first focusing on the diagonal part. 
In the following, we will require information on the operator family $S_*Q^\lambda$ with
\begin{equation}\label{Q}
Q^\lambda:=\int_0^{a_m} b(a)\, U_{A}^\lambda(a,0)\,\rd a\,,\quad \lambda\in \C\,,
\end{equation}
where
$$
U_{A}^\lambda(a,\sigma):=e^{-\lambda (a-\sigma)}\,U_{A}(a,\sigma)\,,\quad 0\le \sigma\le a\le a_m\,.
$$
It follows from \eqref{EST} and \eqref{A} that
\begin{equation}\label{QQ}
S_*Q^\lambda\in\mathcal{L}\big(L_p(\Omega),W_{p,\mathcal{B}}^{1}(\Omega)\big) \cap \mathcal{K}\big(L_p(\Omega)\big)\,,
\end{equation} 
where the compact embedding of $W_{p,\mathcal{B}}^{1}(\Omega)$ into $L_p(\Omega)$ ensures the compactness of the operator  $S_*Q^\lambda$ on $L_p(\Omega)$. Moreover, for $\lambda\in \R$, the operator $S_*Q^\lambda\in\mathcal{L}(L_p(\Omega))$ is an irreducible operator for $p>n$ \cite[Corollary 13.6]{DanersKochMedina}. Its spectral radius is thus characterized by the Krein-Rutman Theorem. We cite the following result in this context:

\begin{lem}[{\bf \cite[Lemma 2.4, Lemma 2.5]{WalkerMOFM}}]\label{L0} 
For $\lambda\in \R$,  the spectral radius $r(S_*Q^\lambda)$ is positive and a simple eigenvalue of \mbox{$S_*Q^\lambda\in\mathcal{K}(L_p(\Omega))$} with an eigenvector $\zeta_\lambda \in W_{p,\mathcal{B}}^1(\Omega)$ that is quasi-interior in~$L_p^+(\Omega)$. It is the only eigenvalue of $S_*Q^\lambda$ with a positive eigenvector.
The mapping  
$$
\R\rightarrow (0,\infty)\,,\quad \lambda\mapsto r(S_*Q^\lambda)
$$ 
is continuous and strictly decreasing with 
$$
\lim_{\lambda\rightarrow-\infty}r(S_*Q^\lambda) =\infty\,,\qquad  \lim_{\lambda\rightarrow\infty}r(S_*Q^\lambda) =0\,.
$$
\end{lem}

Now, in order to introduce the semigroup associated with~\eqref{C}, we recall from~\cite[Lemma 5.1]{WalkerIUMJ} that there exists a mapping 
\begin{equation*}
B: [(S_0,I_0)\mapsto B_{(S_0,I_0)}]\in\ml \big(L_p(\Omega)\times L_1(J,L_p(\Omega)),
C(\R^+,L_p(\Omega))\big)
\end{equation*}
such that $B=B_{(S_0,I_0)}$ is for given $(S_0,I_0)\in L_p(\Omega)\times L_1(J,L_p(\Omega))$ the unique solution to
the Volterra equation\footnote{We recall again that  for $t>a_m$, integral $\int_0^t\rd a$ equal $\int_0^{a_m}\rd a$ and integrals $\int_0^{a_m-t}\rd a$ vanish.} 
\begin{subequations}\label{100}
    \begin{equation}\label{500}
    B(t)=S_*\int_0^t \, b(a)\, U_{A}(a,0)\, B(t-a)\, \rd
    a + S_*\int_0^{a_m-t} b(a+t)\, U_{A}(a+t,a)\, I_0(a)\, \rd a +q_*e^{t A_1^*} S_0\, ,
    \end{equation}
for $t\ge 0$. If $(S_0,I_0)\in L_p^+(\Omega)\times L_1^+(J,L_p(\Omega))$, then $B_{(S_0,I_0)}(t)\in L_p^+(\Omega)$ for $t\ge 0$. Now, given $(S_0,I_0)\in L_p(\Omega)\times L_1(J,L_p(\Omega))$, define
  \begin{equation}\label{501}
     \big[\I_*(t) (S_0,I_0)\big](a)\, :=\, \left\{ \begin{aligned}
    &U_{A}(a,a-t)\, I_0(a-t)\, ,& &   (a,t)\in J\times \R^+\,,\quad  t\le a\,, \\
    & U_{A}(a,0)\, B_{(S_0,I_0)}(t-a)\, ,& & (a,t)\in J\times \R^+\,,\quad   t>a\, , 
    \end{aligned}
   \right.
    \end{equation}	
and
\begin{equation}\label{SG}
     \T_*(t) (S_0,I_0)\, :=\, \big(e^{t A_1^*} S_0\,,\, \I_*(t) (S_0,I_0)\big)\, , \qquad t\ge 0\,.
    \end{equation}
\end{subequations}	
Then $(\T_*(t))_{t\ge 0}$ defines a strongly continuous semigroup on $L_p(\Omega)\times L_1(J,L_p(\Omega))$:

\begin{thm}\label{T0}
Suppose  \eqref{A} and let $(S_*,I_*)$ be a steady state to~\eqref{E} satisfying~\eqref{ss}. Define $(\T_*(t))_{t\ge 0}$ on $L_p(\Omega)\times L_1(J,L_p(\Omega))$ with $p>n$ according to~\eqref{100}. \vspace{2mm}

{\bf (a)} $(\T_*(t))_{t\ge 0}$ is a strongly continuous, eventually compact, positive semigroup on the space $ L_p(\Omega)\times L_1(J,L_p(\Omega))$.\vspace{2mm}

{\bf (b)} Let $\A_*$ be the infinitesimal generator of the semigroup~$(\T_*(t))_{t\ge 0}$.  Then $(\phi,\psi)\in \dom(\A_*)$ if and only if $(\phi,\psi)\in W_{p,\mathcal{B}}^{2}(\Omega)\times C(J,L_p(\Omega))$ and there exists $\zeta\in L_1(J,L_p(\Omega))$ such that $\psi$ is the mild solution to
\begin{subequations}\label{psi}
\begin{equation}\label{psi1}
\partial_a\psi = A(a)\psi +\zeta(a)\,,\quad a\in J\,,\qquad \psi(0)=S_*\int_0^{a_m} b(a) \psi(a)\,\rd a + q_*\phi\,.
\end{equation}
In this case, 
\begin{equation}\label{psi2}
\A_* ( \phi  , \psi) = \left( A_1^*\phi  , -\zeta \right)\,.
\end{equation}
\end{subequations}

{\bf (c)}  $\A_*$ has compact resolvent.\vspace{2mm}

{\bf (d)} The spectral bound $s(\A_*)$ and the growth bound $\omega_0(\T_*)$ are equal.

\end{thm}

\begin{proof}
\noindent {\bf (a) Semigroup:} One may follow the lines of the proof of \cite[Theorem 4]{WebbSpringer} to show that $(\T_*(t))_{t\ge 0}$ defines a strongly continuous  positive  semigroup on $L_p(\Omega)\times L_1(J,L_p(\Omega))$ (see also \cite{WalkerMOFM,WalkerIUMJ}). That this semigroup is eventually compact follows as in \cite[Theorem~1.2~(a)]{WalkerIUMJ} invoking  Kolmogorov's compactness criterion and using the compact embedding of  $W_{p,\mathcal{B}}^2(\Omega)$ into  $L_p(\Omega)$.\vspace{2mm}

\noindent {\bf (b) Generator:} The proof of the characterization of the generator $\A_*$ is mostly along the lines of the proof of \cite[Theorem~1.4~(a)]{WalkerIUMJ}, though a bit tricky. We provide some details here. The key is to derive a description of the resolvent $(\lambda-\A_*)^{-1}$.

{\bf (i)}  To this end, we fix $\lambda>0$ large enough (in particular in the resolvent set of $\A_*$) and write
$$
(\lambda-\A_*)^{-1}(S_0,I_0)=(\phi,\psi)\in L_p(\Omega)\times L_1(J,L_p(\Omega))
$$
for $(S_0,I_0)\in L_p(\Omega)\times L_1(J,L_p(\Omega))$. Then, using the Laplace transform formula 
$$
(\lambda-\A_*)^{-1}(S_0,I_0)=\int_0^\infty e^{-\lambda t}\, \T_*(t) (S_0,I_0)\, \rd t
$$
and recalling \eqref{SG}, we readily obtain
\begin{equation*}%
\phi=\int_0^\infty e^{-\lambda t}\, e^{tA_1^*} S_0\, \rd t=(\lambda-A_1^*)^{-1}S_0
\end{equation*}
and using \eqref{501}, for  $a\in J$,
\begin{equation}
\begin{split}\label{eq2}
\psi(a)&=\int_0^\infty e^{-\lambda t}\, \big[\I_*(t) (S_0,I_0)\big](a)\, \rd t\\
& =\int_0^a U_{A}^\lambda(a,t)\,I_0(t)\, \rd t +U_{A}^\lambda(a,0) \int_0^\infty e^{-\lambda t} B_{(S_0,I_0)}(t)\, \rd t\,.  
\end{split}
\end{equation}
Invoking \eqref{eq2}, \eqref{500}, and \eqref{Q} we derive  in particular
that
\begin{equation*}
\begin{split}
\psi(0)&=\int_0^\infty e^{-\lambda t} B_{(S_0,I_0)}(t)\, \rd t\\
&=\int_0^\infty e^{-\lambda t}\, S_*\int_0^{\min\{t,a_m\}} \, b(a)\, U_{A}(a,0)\, B_{(S_0,I_0)}(t-a)\, \rd
    a\,\rd t\\
&\quad + \int_0^{a_m} e^{-\lambda t}S_*\int_0^{a_m-t} b(a+t)\, U_{A}(a+t,a)\, I_0(a)\, \rd a\,\rd t +\int_0^\infty e^{-\lambda t}q_*e^{t A_1^*} S_0\,\rd t\\
&= S_*\int_0^{a_m} b(a)\, U_{A}^\lambda(a,0)\,\rd a \,\psi(0) + S_*\int_0^{a_m} b(a)\int_0^a U_{A}^\lambda(a,t)\, I_0(t) \,\rd t\,\rd a +q_*(\lambda-A_1^*)^{-1}S_0\\
&= S_*Q^\lambda\psi(0) + S_*\int_0^{a_m} b(a)\int_0^a U_{A}^\lambda(a,t)\, I_0(t) \,\rd t\,\rd a +q_*\phi\,.
\end{split}
\end{equation*}
Summarizing, we have shown that if $(S_0,I_0)\in L_p(\Omega)\times L_1(J,L_p(\Omega))$ and
$$
(\phi,\psi)=(\lambda-\A_*)^{-1}(S_0,I_0)
$$
for $\lambda>0$ large enough, then
\begin{subequations}\label{p0}
\begin{equation}\label{p1a}
\phi=(\lambda-A_1^*)^{-1}S_0\in W_{p,\mathcal{B}}^2(\Omega)
\end{equation}
and $\psi\in C(J,L_p(\Omega))$ is a mild solution to
\begin{equation}\label{p1}
\partial_a\psi =(-\lambda+A(a))\psi +I_0(a)\,,\quad a\in J\,, 
\end{equation}
subject to
\begin{equation}\label{p2}
\begin{split}
 \psi(0)&= S_*\int_0^{a_m} b(a) \psi(a)\,\rd a+q_*\phi\\
&=S_*Q^\lambda\psi(0) + S_*\int_0^{a_m} b(a)\int_0^a U_{A}^\lambda(a,t)\, I_0(t) \,\rd t\,\rd a +q_*\phi\,.
\end{split}
\end{equation}
\end{subequations}

{\bf (ii)} Consider now an arbitrary $(\phi,\psi)\in\dom (\A_*)\subset L_p(\Omega)\times L_1(J,L_p(\Omega))$. Defining (for $\lambda>0$ large enough)
\begin{equation*}
(S_0,I_0):=(\lambda-\A_*)(\phi,\psi)\in L_p(\Omega)\times L_1(J,L_p(\Omega))\,,
\end{equation*}
it readily follows from  \eqref{p0}  that
\begin{equation*}
\phi=(\lambda-A_1^*)^{-1}S_0\in W_{p,\mathcal{B}}^{2}(\Omega)
\end{equation*}
while $\psi\in C(J,L_p(\Omega))$ is a mild solution to
\begin{equation*}
\partial_a\psi =(-\lambda+A(a))\psi +I_0(a)=A(a)\psi +\zeta(a)\,,\quad a\in J\,,\qquad 
\end{equation*}
with $\zeta:=I_0-\lambda\psi\in L_1(J,L_p(\Omega))$ and
subject to
\begin{equation}\label{p2a}
 \psi(0)= S_*\int_0^{a_m} b(a) \psi(a)\,\rd a+q_*\phi\,.
\end{equation}
This is \eqref{psi1}, while \eqref{psi2} follows from $$\A_*(\phi,\psi)=\lambda (\phi,\psi)-(S_0,I_0)=(A_1^*\phi, -\zeta)\,.$$

{\bf (iii)} Conversely, consider $(\phi,\psi)\in W_{p,\mathcal{B}}^{2}(\Omega)\times C(J,L_p(\Omega))$ with the property that there exists $\zeta\in L_1(J,L_p(\Omega))$ such that $\psi$ is the mild solution to
\begin{equation*}
\partial_a\psi = A(a)\psi +\zeta(a)\,,\quad a\in J\,,\qquad \psi(0)=S_*\int_0^{a_m} b(a) \psi(a)\,\rd a + q_*\phi\,.
\end{equation*}
Thus, for $\lambda>0$ large enough and
$$
I_0:=\lambda\psi+\zeta\in L_1(J,L_p(\Omega))\,,
$$
we see that $\psi\in C(J,L_p(\Omega))$ is the mild solution to
\begin{equation*}
\partial_a\psi = (-\lambda+A(a))\psi +I_0(a)\,,\quad a\in J\,,\qquad \psi(0)=S_*\int_0^{a_m} b(a) \psi(a)\,\rd a + q_*\phi\,,
\end{equation*}
and thus satisfies
\begin{equation}\label{o}
\psi(0)=S_*Q^\lambda \psi(0) + S_*\int_0^{a_m} b(a)\int_0^a U_{A}^\lambda(a,t)\, I_0(t) \,\rd t\,\rd a +q_*\phi\,.
\end{equation}
Define now
$$
S_0:=(\lambda -A_1^*)\phi\in L_p(\Omega)
$$ 
and
$$
(\bar\phi,\bar\psi):=(\lambda-\A_*)^{-1}(S_0,I_0)\in\dom(\A_*)\,.
$$
Then, according to~\eqref{p0},
\begin{equation*}
\bar\phi=(\lambda-A_1^*)^{-1}S_0=\phi
\end{equation*}
while $\bar\psi\in C(J,L_p(\Omega))$ is the mild solution to
$$
\partial_a\bar\psi =(-\lambda+A(a))\bar\psi+I_0(a)\,,\quad a\in J\,,
$$
subject to
\begin{align}\label{oo}
 \bar\psi(0)&=S_*Q^\lambda\bar\psi(0) + S_*\int_0^{a_m} b(a)\int_0^a U_{A}^\lambda(a,t)\, I_0(t) \,\rd t\,\rd a +q_*\bar\phi\,.
\end{align}
Since $\bar\phi=\phi$, it follows from \eqref{o} and \eqref{oo} that
$$
(1-S_*Q^\lambda)\bar\psi(0)=(1-S_*Q^\lambda)\psi(0)
$$
and hence $\bar\psi(0)=\psi(0)$ according to Lemma~\ref{L0} for $\lambda>0$ large enough (so that $r(S_*Q^\lambda)<1$). Consequently, $\bar\psi=\psi$ and therefore 
$$
(\phi,\psi)=(\bar\phi,\bar\psi)\in \dom(\A_*)\,.
$$ 
This proves part {\bf (b)}.\vspace{2mm}

\noindent {\bf (c) Compact Resolvent:} In order to show that  $\A_*$ has compact resolvent, let $\lambda>0$ again be sufficiently large (i.e. $\lambda$ in the resolvent set of $\A_*$ and $r(S_*Q^\lambda)<1$). Let $(S_{0,j},I_{0,j})_{j\in\N}$ be a bounded sequence in $L_p(\Omega)\times L_1(J,L_p(\Omega))$ and set
$$
(\phi_j,\psi_j):=(\lambda-\A_*)^{-1}(S_{0,j},I_{0,j})\,.
$$
Then \eqref{p1a} yields $\phi_j=(\lambda-A_1^*)^{-1}S_{0,j}$ so that $(\phi_j)_{j\in\N}$ is a bounded sequence in $W_{p,\mathcal{B}}^2(\Omega)$, the latter being compactly embedded in $L_p(\Omega)$. It remains to show that $(\psi_j)_{j\in\N}$ is relatively compact in $L_1(J,L_p(\Omega))$ for which we first note from \eqref{p1}-\eqref{p2} that
\begin{equation}\label{p10}
\psi_j(a) =U_A^\lambda(a,0)\psi_j(0)+\int_0^aU_A^\lambda(a,\sigma) I_{0,j}(\sigma)\,\rd \sigma\,,\quad a\in J\,,
\end{equation}
and
\begin{equation}\label{p20}
 (1-S_*Q^\lambda)\psi_j(0)=  S_*\int_0^{a_m} b(a)\int_0^a U_{A}^\lambda(a,\sigma)\, I_{0,j}(\sigma) \,\rd \sigma\,\rd a +q_*\phi_j\,.
\end{equation}
In particular, since by \eqref{A3} and \eqref{EST} 
\begin{align*}
\Bigg\|\int_0^{a_m} b(a)&\int_0^a U_{A}^\lambda(a,\sigma)\, I_{0,j}(\sigma) \,\rd \sigma\,\rd a \Bigg\|_{W_{p,\mathcal{B}}^{1}(\Omega)}
\\
&\le c \|b\|_{L_\infty(J,C^1(\Omega))}\, \int_0^{a_m} \int_0^a (a-\sigma)^{-1/2}\, \|I_{0,j}(\sigma) \|_{L_p(\Omega)}\,\rd \sigma\,\rd a\\
&\le c_1 \int_0^{a_m} \|I_{0,j}(\sigma) \|_{L_p(\Omega)} \int_\sigma^{a_m} (a-\sigma)^{-1/2} \,\rd a\,\rd \sigma\le c_2 \|I_{0,j} \|_{L_1(J,L_p(\Omega))}\,,
\end{align*}
the sequence
$$
\left(\int_0^{a_m} b(a)\int_0^a U_{A}^\lambda(a,\sigma)\, I_{0,j}(\sigma) \,\rd \sigma\,\rd a\right)_{j\in\N}
$$
is bounded in $W_{p,\mathcal{B}}^{1}(\Omega)$, and
we thus deduce from \eqref{p20}, \eqref{QQ}, and $r(S_*Q^\lambda)<1$ that 
\begin{equation}\label{pq}
(\psi_j(0))_{j\in\N}\ \text{  is bounded in}\ W_{p,\mathcal{B}}^{1}(\Omega)\,.
\end{equation} 
Setting
$$
u_j(a):=\int_0^aU_A^\lambda(a,\sigma) I_{0,j}(\sigma)\,\rd \sigma\,,\quad a\in J\,,\quad j\in\N\,,
$$
we next show that $\{u_j\,;\, j\in\N\}$ is relatively compact in $L_1(J,L_p(\Omega))$ adopting the arguments from~\cite{BarasHassanVeron} (there the case of semigroups was considered): \vspace{2mm}

{\bf (i)}  We first fix $\mu>0$ and define
$$
v_j^\mu(a):=U_A^\lambda(\mu+a,a)u_j(a)=\int_0^aU_A^\lambda(\mu+a,\sigma) I_{0,j}(\sigma)\,\rd \sigma\,,\quad a\in J\,,\quad j\in\N\,.
$$
Since 
$$
\|u_j(a)\|_{L_p(\Omega)}\le \|I_{0,j}\|_{L_1(J,L_p(\Omega))}\,, \quad a\in J\,,\quad j\in\N\,,
$$
and 
$$
U_A^\lambda(\mu+a,a)\in\ml\big(L_p(\Omega),W_{p,\mathcal{B}}^2(\Omega)\big)\subset \mathcal{K}\big(L_p(\Omega)\big)\,,\quad \quad a\in J\,,
$$
we see that, for every $a\in J$, the sequence $(v_j^\mu(a))_{j\in\N}$ is relatively compact in $L_p(\Omega)$. Next, in order to check equi-integrability we recall from \cite[II.~Equation~(2.1.2)]{LQPP} that
$$
U_A^\lambda\in C\big(\Delta_J^*,\ml(L_p(\Omega))\big) \ \text{ with }\ \Delta_J^*:=\{(a,\sigma)\,;\, 0\le \sigma< a\le a_m\}\,,
$$
and note for $\xi>0$ that the set 
$$
K_\xi:=\{(a,\sigma)\in \Delta_J^*\,;\,  \sigma+\xi\le  a\}
$$
is compact in $\Delta_J^*$. We thus find for every $\ve>0$ and $\xi>0$ some $\eta>0$ such that
$$
\|U_A(a_1,\sigma_1)-U_A(a_2,\sigma_2)\|_{\ml(L_p(\Omega))}\le \ve\,,\qquad (a_i,\sigma_i)\in K_\xi\,,\quad \vert (a_1,\sigma_1)-(a_2,\sigma_2)\vert\le \eta\,.
$$
Taking $\ve>0$ arbitrary and $\xi=\mu>0$, we use \eqref{EST2} to derive, for $h\in (0,\eta)$ with  $0<a\le a+h\le a_m$,
\begin{align*}
\|v_j^\mu(a+h)-v_j^\mu(a)\|_{L_p(\Omega)}&\le
 \int_a^{a+h}\|U_A^\lambda(\mu+a+h,\sigma)\|_{\ml(L_p(\Omega))}\,\|I_{0,j}(\sigma)\|_{L_p(\Omega)}\,\rd\sigma\\
&\quad + \int_0^{a}\|U_A^\lambda(\mu+a+h,\sigma)-U_A^\lambda(\mu+a,\sigma)\|_{\ml(L_p(\Omega))}\,\|I_{0,j}(\sigma)\|_{L_p(\Omega)}\,\rd\sigma\\
&\le \int_a^{a+h}\|I_{0,j}(\sigma)\|_{L_p(\Omega)}\,\rd\sigma+\ve  \int_0^{a}\|I_{0,j}(\sigma)\|_{L_p(\Omega)}\,\rd\sigma
\end{align*}
and therefore
\begin{align*}
\int_0^{a_m}\|\tilde v_j^\mu(a+h)-\tilde v_j^\mu(a)\|_{L_p(\Omega)}\,\rd a &\le
 h \|I_{0,j} \|_{L_1(J,L_p(\Omega))}+\ve a_m \|I_{0,j} \|_{L_1(J,L_p(\Omega))}\,,
\end{align*}
where the tilde refers to the trivial extension. Hence, 
\begin{align*}
\lim_{h\to 0}\, \sup_{j\in\N}\int_0^{a_m}\|\tilde v_j^\mu(a+h)-\tilde v_j^\mu(a)\|_{L_p(\Omega)}\,\rd a =0
\end{align*} 
so that $\{v_j^\mu\,;\, j\in\N\}$  is equi-integrable and thus
\begin{align}\label{equi}
\{v_j^\mu\,;\, j\in\N\}\ \text{  is relatively compact in  $L_1(J,L_p(\Omega))$ for $\mu>0$}\,.
\end{align} 

{\bf (ii)} We next consider the limit $\mu\to 0$. Given  $\ve>0$, $\xi>0$, and using the notation from the previous part we have, for $a\in J$ with $a\ge \xi$ and $0<\mu<\eta$,
\begin{align*}
\|v_j^\mu(a)-u_j(a)\|_{L_p(\Omega)}&\le
 \int_0^{a-\xi}\|U_A^\lambda(\mu+a,\sigma)-U_A^\lambda(a,\sigma)\|_{\ml(L_p(\Omega))}\,\|I_{0,j}(\sigma)\|_{L_p(\Omega)}\,\rd\sigma\\
&\quad +  \int_{a-\xi}^a \big(\|U_A^\lambda(\mu+a,\sigma)\|_{\ml(L_p(\Omega))}+\|U_A^\lambda(a,\sigma)\|_{\ml(L_p(\Omega))}\big)\,\|I_{0,j}(\sigma)\|_{L_p(\Omega)}\,\rd\sigma\\
&\le \ve \int_0^{a-\xi} \|I_{0,j}(\sigma)\|_{L_p(\Omega)}\,\rd\sigma + 2\int_{a-\xi}^a \|I_{0,j}(\sigma)\|_{L_p(\Omega)}\,\rd\sigma\,, 
\end{align*}
while for $0\le a\le \xi$ we have
\begin{align*}
\|v_j^\mu(a)-u_j(a)\|_{L_p(\Omega)}&\le  2\int_0^{\xi} \|I_{0,j}(\sigma)\|_{L_p(\Omega)}\,\rd\sigma\,.
\end{align*}
Therefore,
\begin{align*}
\|v_j^\mu-u_j\|_{L_1(J,L_p(\Omega))}&\le  4 \xi \|I_{0,j}\|_{L_1(J,L_p(\Omega))}+\ve a_m \|I_{0,j}\|_{L_1(J,L_p(\Omega))}
\end{align*}
so that
\begin{align*}
\lim_{\mu\to 0}\, \sup_{j\in\N}\, \|v_j^\mu-u_j\|_{L_1(J,L_p(\Omega))}=0\,.
\end{align*}
Together with \eqref{equi} we conclude that $\{u_j\,;\, j\in\N\}$ is relatively compact in $L_1(J,L_p(\Omega))$.\vspace{2mm}

{\bf (iii)} Finally,  since  
$$
\|U_A^\lambda (a+h,0)-U_A^\lambda (a,0)\|_{\ml(W_{p,\mathcal{B}}^{1}(\Omega),L_p(\Omega))}\le c h^{1/4}\,,\quad 0\le a\le a+h\le a_m\,,
$$
according to \cite[II.~Equation~(5.3.8]{LQPP}, it readily follows from \eqref{pq} and the Arzel\`a - Ascoli Theorem that $(U_A^\lambda (\cdot,0)\psi_j(0))_{j\in\N}$ is relatively compact in $C(J,L_p(\Omega))$. Consequently, we deduce from the previous step {\bf (ii)} and~\eqref{p10} that $(\psi_j)_{j\in\N}$ is relatively compact in $L_1(J,L_p(\Omega))$. Therefore, $\A_*$ has indeed a compact resolvent.\vspace{2mm}

\noindent {\bf (d) Spectral Bound:} Since $\T_*$ is eventually compact, it follows from \cite[IV.~Corollary 3.12]{EngelNagel} that $s(\A_*)=\omega_0(\T_*)$.

\end{proof}

Introducing now the perturbation
$$
\B_*:=\left(\begin{matrix}0 & -P_* +N\\ 0&0 \end{matrix}\right)\in \ml\big(L_p(\Omega)\times L_1(J,L_p(\Omega))\big)
$$
with $P_*$ and $N$ defined in \eqref{qnot}
and observing from \eqref{psi} that $\A_*+\B_*$ is exactly the linearized operator appearing in \eqref{C1} subject to~\eqref{C2}, we obtain  the desired generation result for the linearization:

\begin{cor}\label{C43}
Let the assumptions of Theorem~\ref{T0} be satisfied. Then, the generator 
 $\A_*+\B_*$ of the strongly continuous semigroup $(\mS_*(t))_{t\ge 0}$ on $L_p(\Omega)\times L_1(J,L_p(\Omega))$, associated with the linearized problem~\eqref{C}, has compact resolvent. In particular, the spectrum $\sigma(\A_*+\B_*)$ is a pure point spectrum without finite accumulation point.
\end{cor}

\begin{proof}
This follows from Theorem~\ref{T0} and \cite[III.~Proposition~1.12]{EngelNagel}.  
\end{proof}

\begin{deff}\label{DefStable}
We call the steady state $(S_*,I_*)$ \emph{linearly stable} in $L_p(\Omega)\times L_1(J,L_p(\Omega))$, if the generator $\A_*+\B_*$ of the linearized problem satisfies $\mathrm{Re}\, \sigma(\A_*+\B_*)<0$, while we call $(S_*,I_*)$ \emph{linearly unstable} if $ \sigma(\A_*+\B_*)\cap [\mathrm{Re}\,\lambda>0]\not=\emptyset$. 
\end{deff}

\begin{rem}\label{R2} 
Corollary~\ref{C43} implies that the linear stability of the steady state $(S_*,I_*)$ is determined from (the real parts of) those $\lambda\in \C$ for which there is a nontrivial \mbox{$(S,I)\in W_{p,\mathcal{B}}^{2}(\Omega)\times C(J,L_p(\Omega))$} satisfying
\begin{align*}
\left(\begin{matrix} A_1^* & -P_* +N \\ 0& -\partial_a+A(a) \end{matrix}\right)\left(\begin{matrix} S \\ I \end{matrix}\right) =\lambda \left(\begin{matrix} S \\ I \end{matrix}\right)\,,\qquad I(0)-P_*I=q_*S\,,
\end{align*}
(equality for the second component in the sense of Theorem~\ref{T0}) 
with notation introduced in \eqref{STAR}.
\end{rem}


{  \begin{rem}\label{R1a}
Since the semigroup $(e^{t\A_*})_{t\ge 0}$ is eventually compact and due to the particular (nonlocal) form of the perturbation $\B_*$, one can in fact show that the perturbation semigroup $\mS_*=(e^{t(\A_*+\B_*)})_{t\ge 0}$ is also eventually compact. Consequently,
$$
\omega(\mS_*)=s(\A_*+\B_*)\,,
$$
that is, the growth bound of the semigroup $\mS_*$ and the spectral bound of its generator $\A_*+\B_*$ coincide. A steady state $(S_*,I_*)$ is thus linearly stable in the sense of Definition~\ref{DefStable} if and only if the semigroup $\mS_*$ associated with the linearization around this steady state has an exponential decay. One can further prove that this indeed implies the asymptotic stability of the steady state. The technical details of the proof follow along the lines of \cite{WalkerInstabiliy} (see also \cite{WalkerZehetbauer}).
\end{rem}
}

{ 
\section{Linearized Stability: Proof of Theorem~\ref{TX}}\label{Sec4B}

We shall apply the results from the previous section. In the following, we are still imposing~\eqref{A} with $p>(2\vee n)$ and assume for simplicity~\eqref{Aa}. 
Recall that $\kappa_1, \kappa_2>0$. We consider only steady states $(S_*,I_*)$ to \eqref{E} with regularity as in~\eqref{ss}.

For the investigation of stability of steady states  recall that the spectrum of the Laplacian is (counted according to multiplicity)
$$
 \sigma(-\Delta_\mathcal{B})= \{\mu_0,\mu_1,\mu_2,\ldots\}
$$
with $0\le \mu_i\le \mu_{i+1}$ for $i\ge 0$. In fact, $\mu_0>0$ in the Dirichlet case $\delta=0$ and  $\mu_0=0$ in the Neumann case~$\delta=1$.

\subsection{The Trivial Steady State}

The linear stability of the trivial steady state $(S_*,I_*)=(0,0)$  depends on the sign of $\kappa_1-\mu_0$ (which is always positive for the Neumann case $\delta=1$ but may be negative in the Dirichlet case $\delta=0$):

\begin{prop}\label{P99x}
The trivial steady state $(S_*,I_*)=(0,0)$ is linearly  unstable if $\kappa_1>\mu_0$ and linearly stable if $\kappa_1<\mu_0$.
\end{prop}

\begin{proof}
Using  the notation introduced in~\eqref{STAR} for $(S_*,I_*)=(0,0)$, we have
$$
q_*=0\,,\qquad P_*=0\,,\qquad A_1^*=\Delta_\mathcal{B}+\kappa_1\,,
$$
so that Remark~\ref{R2} leads to investigating the eigenvalue problem
\begin{align*} 
\left(\begin{matrix} \Delta_\mathcal{B}+\kappa_1  & 0 \\ 0& -\partial_a+A(a)\end{matrix}\right)\left(\begin{matrix} S \\ I \end{matrix}\right) =\lambda \left(\begin{matrix} S \\ I \end{matrix}\right)\,,\qquad I(0)=0\,,
\end{align*}
for a  nontrivial $(S,I)\in W_{p,\mathcal{B}}^{2}(\Omega)\times C(J,L_p(\Omega))$; that is,
\begin{align*}
-\Delta_\mathcal{B}S&=(\kappa_1-\lambda) S\,,   \\ 
\partial_aI(a)&=\big(-\lambda+A(a)\big)I(a) \,,\quad a\in J\,,\qquad I(0)=0\,,
\end{align*}
with mild solution $I$. Hence $I=0$ so that $S\not=0$ and thus $\kappa_1-\lambda\in \sigma(-\Delta_\mathcal{B})$.
Consequently, $s(\A_*)=\kappa_1-\mu_0$ is an eigenvalue. 
\end{proof}

\subsection{The Disease-Free Steady State}

The existence of a disease-free steady state $(S_*,I_*)=( \tilde S_*,0)$ reduces to finding a positive non-trivial solution $ \tilde S_*\in W_{p,\mathcal{B}}^2(\Omega)$ to the semilinear equation
\begin{equation}\label{ST}
-\Delta_\mathcal{B} \tilde S_*+\frac{\kappa_1}{\kappa_2}  \tilde S_*^2=\kappa_1  \tilde S_*\,.
\end{equation}
Clearly, in the Neumann case $\delta=1$, the positive (constant) solution to \eqref{ST} is $\tilde S_*=\kappa_2$. \\

To continue let us recall the following result from \cite[Theorem~12,~Theorem~16]{AmannMaxPrinc}:

\begin{lem}\label{Amann}
Let $q\in L_\infty(\Omega)$. Then the eigenvalue problem
\begin{equation*}\label{EVP}
-\Delta_\mathcal{B} u+q u=\lambda u
\end{equation*}
has a smallest eigenvalue $\lambda=\lambda_0(q)\in \R$ (in the sense that $\mathrm{Re}\, \lambda>\lambda_0(q)$ for every other eigenvalue~$\lambda$). This principal eigenvalue is simple and the only eigenvalue with a positive eigenfunction~$u_0$, i.e. $u_0\in W_{p,\mathcal{B}}^2(\Omega)$ for every $p\in (1,\infty)$ and $u_0>0$ in $\Omega$. Moreover, if $q_1,q_2\in L_\infty(\Omega)$ with $q_1\le q_2$ and $q_1\not\equiv q_2$, then $\lambda_0(q_1)<\lambda_0( q_2)$. In fact, $\lambda_0(0)=\mu_0$.
\end{lem}

Now, if $\tilde S_*\in W_{p,\mathcal{B}}^2(\Omega) $ is a positive non-trivial solution to~\eqref{ST}, then necessarily
\begin{equation}\label{k1}
\kappa_1=\lambda_0\left(\frac{\kappa_1 \tilde S_*}{\kappa_2} \right)>\lambda_0(0)=\mu_0\,.
\end{equation}
Conversely, it  follows from \cite{BlatBrown} that if $\kappa_1>\lambda_0(0)=\mu_0$, then~\eqref{ST} admits a positive solution $\tilde S_*\in W_{p,\mathcal{B}}^2(\Omega)$ (see also~\cite{WalkerCrelle}). This solution is unique. Indeed, if there was another positive solution $\hat S_*\in W_{p,\mathcal{B}}^2(\Omega)$ to~\eqref{ST}, then $z:=\tilde S_*-\hat S_*$ solves the eigenvalue equation
\begin{equation*}
-\Delta_\mathcal{B} z+\frac{\kappa_1}{\kappa_2}(\tilde S_*+\hat S_*) z=\kappa_1 z
\end{equation*}
so that
$$
\kappa_1\ge \lambda_0\left(\frac{\kappa_1}{\kappa_2}(\tilde S_*+\hat S_*)\right)\,.
$$
According to \eqref{k1} and the monotonicity of $\lambda_0(q)$ with respect to $q$ this yields the contradiction
$$
\kappa_1=\lambda_0\left(\frac{\kappa_1 \tilde S_*}{\kappa_2} \right)<\lambda_0\left(\frac{\kappa_1}{\kappa_2}(\tilde S_*+\hat S_*)\right)\le \kappa_1\,.
$$
The linear stability of the disease-free steady state $(S_*,I_*)=(\tilde S_*,0)$ is then determined from the value of
\begin{subequations}\label{RRR000}
\begin{equation}\label{R0}
\mathsf{R}_0:=r(\tilde S_*Q^0)>0\,,
\end{equation}
where  the family of compact operators
\begin{equation}\label{Qs}
Q^\lambda=\int_0^{a_m} b(a)\, e^{-\lambda a}\, U_{A}(a,0)\,\rd a\in\mathcal{K}(L_p(\Omega))\,,\quad\lambda\in \C\,,
\end{equation}
\end{subequations}
was introduced in~\eqref{Q} and properties of the spectral radius $r(\tilde S_*Q^\lambda)$ are stated in~Lemma~\ref{L0} for~$\lambda\in\R$.

\begin{prop}\label{P100x}
There is a disease-free steady state $(S_*,I_*)=(\tilde S_*,0)$ with a smooth function $\tilde S_*>0$ if and only if $\kappa_1>\mu_0$.  In this case, $\tilde S_*$ is unique, \eqref{k1} holds, and $\mathsf{R}_0>0$ in \eqref{R0} is well-defined. For Neumann boundary conditions (i.e. $\delta=1$), we have $\tilde S_*=\kappa_2$.

Moreover, $(S_*,I_*)=(\tilde S_*,0)$ is linearly stable in \mbox{$L_p(\Omega)\times L_1(J,L_p(\Omega))$} if $\mathsf{R}_0<1$ and linearly unstable if~$\mathsf{R}_0>1$.
\end{prop}

\begin{proof}
We have already shown that  there is a (unique) disease-free steady state $(S_*,I_*)=(\tilde S_*,0)$ with $\tilde S_*>0$ (satisfying \eqref{k1}) if and only if $\kappa_1>\mu_0$. Thus, let $\kappa_1>\mu_0$. According to Remark~\ref{R2} we have to check the real parts of solutions $\lambda$ to the eigenvalue problem
\begin{align*} 
\left(\begin{matrix} A_1^*  & -P_* \\ 0& -\partial_a+A(a) \end{matrix}\right)\left(\begin{matrix} S \\ I \end{matrix}\right) =\lambda \left(\begin{matrix} S \\ I \end{matrix}\right)\,,\qquad I(0)=P_*I\,,
\end{align*}
with a  nontrivial $(S,I)\in W_{p,\mathcal{B}}^{2}(\Omega)\times C(J,L_p(\Omega))$, where, due to~\eqref{STAR},
$$
q_*=0\,,\quad P_*I=\tilde S_*\int_0^{a_m}b(a,\cdot)I(a,\cdot)\,\rd a\,,\quad A_1^*=\Delta_\mathcal{B}+\kappa_1-\frac{2\kappa_1 \tilde S_*}{\kappa_2}\,,\quad A(a)=d(a)\Delta_\mathcal{B}-m(a,\cdot)\,.
$$
That is, we have to investigate
\begin{align}
-\Delta_\mathcal{B}S+\frac{2\kappa_1 \tilde S_*}{\kappa_2} S&=(\kappa_1-\lambda)S-\tilde S_*Q^\lambda I(0)\,,  \label{L1} \\ 
\partial_aI&=(-\lambda+A(a))I\,,\quad a\in J\,,\qquad I(0)=\tilde S_*\int_0^{a_m}b(a,\cdot)I(a,\cdot)\,\rd a \,, \label{L2}
\end{align}
where \eqref{L2} entails
\begin{align}\label{L2xxxxx}
I(a)=e^{-\lambda a}U_A(a,0)I(0)\,,\quad a\in J\,,\qquad I(0)=\tilde S_*Q^\lambda I(0)\,.
\end{align}
Assume first that  $\mathsf{R}_0<1$. Then either $I(0)=0$ so that \eqref{L1}, the monotonicity of the principal eigenvalue, and~\eqref{k1} imply
$$\mathrm{Re}\,(\kappa_1-\lambda)\ge \lambda_0\left(\frac{2\kappa_1 \tilde S_*}{\kappa_2}\right)>
\lambda_0\left(\frac{\kappa_1 \tilde S_*}{\kappa_2}\right)=\kappa_1\,,
$$ 
hence $\mathrm{Re}\,\lambda<0$.
Or $I(0)\not=0$ so that~\eqref{L2xxxxx} together with \cite[Theorem~2.3~(b)]{WalkerIUMJ}  imply that $\lambda\in \sigma(\mathbb{A})$, where $\mathbb{A}$ is the generator of a strongly continuous, positive, eventually compact semigroup on $L_1(J,L_p(\Omega))$. Its spectral bound  $s:=s(\mathbb{A})$ is the unique real number with $r(\tilde S_*Q^s)=1$ according to \cite[Proposition 5.2]{WalkerIUMJ}.  Since $\mathsf{R}_0=r(\tilde  S_*Q^0)<1$, it follows from Lemma~\ref{L0} that $s=s(\mathbb{A})<0$ and hence again $\mathrm{Re}\,\lambda<0$. Consequently, if $\mathsf{R}_0<1$, then $(S_*,I_*)=(\tilde S_*,0)$ is linearly stable.

Conversely, if $\mathsf{R}_0=r(\tilde S_*Q^0) >1$, then there is $\lambda>0$ such that $r(\tilde S_*Q^\lambda)=1$ by Lemma~\ref{L0} and there is a nontrivial $I(0)\in W_{p,\mathcal{B}}^1(\Omega)$ with $(1-\tilde S_*Q^\lambda)I(0)=0$. Then 
$$
I(a):= e^{-\lambda a}U_A(a,0)I(0)\,,\quad a\in J\,,
$$ 
satisfies~\eqref{L2}. Moreover, owing to~\eqref{k1}, we have
$$
\kappa_1-\lambda<\kappa_1 =
\lambda_0\left(\frac{\kappa_1 \tilde S_*}{\kappa_2}\right)<\lambda_0\left(\frac{2\kappa_1 \tilde S_*}{\kappa_2}\right)\,,
$$
and thus $\kappa_1-\lambda$ belongs to the resolvent set of $-\Delta_\mathcal{B}+2\kappa_1 \tilde S_*/\kappa_2$. Hence, 
$$
S:=\left( -\Delta_\mathcal{B}+\frac{2\kappa_1 \tilde S_*}{\kappa_2}-\kappa_1+\lambda\right)^{-1}  I(0)\in W_{p,\mathcal{B}}^2(\Omega)
$$
is  a nontrivial solution to~\eqref{L1}. That is, $\lambda >0$ is an eigenvalue
 and the disease-free steady state $(S_*,I_*)=(\tilde S_*,0)$ is thus linearly unstable.
\end{proof}

\subsection{Non-Existence of Endemic Steady States for $\mathsf{R}_0\le 1$}

An endemic steady state $(S_*,I_*)$ is a steady state to~\eqref{E} with $S_*, I_*\ge 0$ and $I_*\not\equiv 0$.
Note that, setting $I_0:=I_*(0)$, this is equivalent to finding a positive $(S_*,I_0)\in W_{p,\mathcal{B}}^2(\Omega)\times W_{p,\mathcal{B}}^1(\Omega)$ with $I_0\not=0$ satisfying
\begin{subequations}\label{SE}
\begin{align}
-\Delta_\mathcal{B}S_*+ Q^0 I_0S_*+\frac{\kappa_1}{\kappa_2}S_*^2&=\kappa_1 S_*\,,  \label{L1yxy} \\ 
I_0&=S_*Q^0 I_0\,. \label{L2yxy}
\end{align}
\end{subequations}
As shown next, $\mathsf{R}_0> 1$ is a necessary  condition for the existence of an endemic state.

\begin{lem}\label{X1}
Let $\kappa_1>\mu_0$ and let $\tilde S_*$ be as in Proposition~\ref{P100x}. If $\mathsf{R}_0=r(\tilde S_*Q^0)\le 1$, then there is no positive solution  $(S_*,I_0)\in W_{p,\mathcal{B}}^2(\Omega)\times W_{p,\mathcal{B}}^1(\Omega)$ to ~\eqref{SE}  with $I_0\not=0$.
\end{lem}

\begin{proof}
Let  $\mathsf{R}_0=r(\tilde S_*Q^0)\le 1$ and assume for contradiction that there was a  positive solution  $(S_*,I_0)\in W_{p,\mathcal{B}}^2(\Omega)\times W_{p,\mathcal{B}}^1(\Omega)$ to ~\eqref{SE}  with $I_0\not=0$.
It follows from \eqref{ST} and \eqref{L1yxy} that $z:=S_*-\tilde S_*$ solves
$$
-\Delta_\mathcal{B}z+ Q^0 I_0z+\frac{\kappa_1}{\kappa_2}(S_*+\tilde S_*)z=\kappa_1 z- I_0\,.
$$
Moreover, we infer from \eqref{L1yxy} and Lemma~\ref{Amann} that
$$
0<\kappa_1=\lambda_0\left(Q^0 I_0+\frac{\kappa_1}{\kappa_2} S_*\right)<\lambda_0\left(Q^0 I_0+\frac{\kappa_1}{\kappa_2} (S_*+\tilde S_*)\right)\,.
$$
Hence, $\kappa_1$ belongs to the resolvent set of the operator 
$$
-\Delta_\mathcal{B}+ Q^0 I_0+\frac{\kappa_1}{\kappa_2}(S_*+\tilde S_*)
$$
and consequently, since $I_0>0$,
$$
z=-\left(\kappa_1-\Delta_\mathcal{B}z+ Q^0 I_0z+\frac{\kappa_1}{\kappa_2}(S_*+\tilde S_*)\right)^{-1}  I_0 \le 0\quad \text{in }\ \Omega\,,
$$
where we used the maximum principle from \cite[Theorem~13]{AmannMaxPrinc}. That is, $S_*\le \tilde S_*$ in $\Omega$ and thus $S_*Q^0\le \tilde S_*Q^0$. The Krein-Rutman Theorem now yields for the spectral radii 
$$
r(S_*Q^0)<r( \tilde S_*Q^0)\le 1\,,
$$
and therefore that the eigenvector equation $p=S_*Q^0p$ has no positive nontrivial solution in contradiction to $I_0>0$ solving~\eqref{L2yxy}.
\end{proof}

\begin{rems}\label{R5}
{\bf (a)} As noted in the proof of Lemma~\ref{X1},   necessary conditions for the existence of an endemic steady state $(S_*,I_*)$  with $S_*, I_*\ge 0$ and $I_*\not\equiv 0$ are (see \eqref{SE}) 
\begin{equation*}
\kappa_1=\lambda_0\left(Q^0 I_*(0)+\frac{\kappa_1}{\kappa_2} S_*\right) \,, \qquad 1=r\left(S_*Q^0\right) <r( \tilde S_*Q^0)=\mathsf{R}_0\,,\qquad S_*\le \tilde S_*\,,
\end{equation*}
where $(\tilde S_*,0)$ is the disease-free steady state.\vspace{1mm}

{\bf (b)} Whether the condition $\mathsf{R}_0>1$ is sufficient for the existence of an endemic steady state is, however, left open. In fact, one can show (see \cite{BlatBrown}) that for every $I_0\in L_p^+(\Omega)$ with $\lambda_0(Q^0I_0)<\kappa_1$ there is a unique solution $S_*=S_*(I_0)\in W_{p,\mathcal{B}}^2(\Omega)$ to~\eqref{L1yxy} with $S_*(I_0)>0$ depending  compactly and smoothly on $I_0$ (by the implicit function theorem), where  $ \tilde S_*=S_*(0)$. This then reduces problem~\eqref{SE} to finding a nontrivial positive fixed point of the smooth, compact operator~$F$ defined by $F(I_0):=S_*(I_0) Q^0 I_0$. Noticing that $DF(0)=\tilde S_* Q^0$ has a positive eigenvector associated with the eigenvalue $\mathsf{R}_0>1$ by the Krein-Rutman Theorem, it would remain to find $\rho>0$ such that $\lambda_0(Q^0I_0)<\kappa_1$ for $\|I_0\|_{L_p}\le \rho$ and $r(S_*(I_0)Q^0)<1$ when $\|I_0\|_{L_p}= \rho$. This then would allow one to apply the fixed point theorem  \cite[Theorem 13.2]{AmannSIAM} to derive the existence of a (unique) positive nontrivial solution $I_0=F(I_0)$ and thus an endemic steady state $(S_*(I_0),I_*)$ with $I_*(a):=U_A(a,0)I_0$.
However, it is open  whether this is indeed possible.

Nevertheless, when considering spatially homogeneous rates and Neumann boundary conditions, there exists a linearly stable endemic state if $\mathsf{R}_0>1$ as stated in Theorem~\ref{T2}.
\end{rems}
}

\section{Linearized Stability in a Particular Model: Proof of Theorem~\ref{T2}}\label{Sec5}

{  For Neumann boundary conditions $\delta=1$ and  spatially homogeneous rates \mbox{$m=m(a)$} and \mbox{$b=b(a)$} we can improve the results from the previous section. Thus assume~\eqref{Aa},~\eqref{AA1}, and $p>(2\vee n)$. Recall that the principal eigenvalue of the Laplacian $-\Delta_N$ subject to Neumann boundary conditions is $\mu_0=0$. We write $W_{p,N}^2(\Omega)$  in the following for the domain of $-\Delta_N$.
Since
\begin{align*}
A(a)=d(a)\Delta_N -m(a)\,,\quad a\in J\,,
\end{align*}
and $m$ is spatially homogeneous, the corresponding evolution operator is given by
\begin{align}\label{QQQ}
U_{A}(a,\sigma)=\exp\left(-\int_\sigma^a m(\tau)\,\rd \tau\right)\,\exp\left(\int_\sigma^a d(\tau)\,\rd \tau\, \Delta_N\right)\,,\quad 0\le \sigma\le a\le a_m\,.
\end{align}
In the previous section we showed that the trivial steady state $(S_*,I_*)=(0,0)$ is linearly unstable, and we have discussed the local linear stability of the disease-free steady state $(\tilde S_*,0)=(\kappa_2,0)$. We next investigate the latter's {\it global} stability.

\subsection{Global Stability of The Disease-Free Steady State $(S_*,I_*)=(\kappa_2,0)$}

We consider \mbox{$S_*=\tilde S_*=\kappa_2$}. Using \eqref{QQQ}, the operators $Q^\lambda$ from \eqref{Qs} become
$$
Q^\lambda= \int_0^{a_m} b(a)\, e^{-\lambda a}\,\Pi(a)\, \exp\left(\int_0^a d(\tau)\,\rd \tau\, \Delta_N\right)\,\rd a\,,\quad \lambda\in\C\,,
$$
where
$$
\Pi(a):=\exp\left(-\int_0^a m(\sigma)\,\rd \sigma\right)\,,\quad a\in J\,.
$$
Noticing 
\begin{align*}
\tilde S_*Q^0{\bf 1}&=\kappa_2\int_0^{a_m} b(a)\,\Pi(a) \,\rd a \, {\bf 1}\,,
\end{align*}
it readily follows from Lemma~\ref{L0} for the spectral radius (see \eqref{R00}) that
\begin{equation} \label{specrad}
\mathsf{R}_0=r(\tilde S_*Q^0)=\kappa_2\int_0^{a_m} b(a) \Pi(a)\, \rd a\,.
\end{equation} 
We have seen in Proposition~\ref{P100x} that $(S_*,I_*)=(\kappa_2,0)$ is  linearly stable when $\mathsf{R}_0 <1$ and linearly unstable when $\mathsf{R}_0 >1$. In the former case, we can prove now its global stability. That is,  if $\mathsf{R}_0 <1$, then any solution to~\eqref{E} subject to positive initial values $(S_0,I_0)$ converges to $(S_*,I_*)=(\kappa_2,0)$.
} 

\begin{prop}\label{P100}
Assume~\eqref{Aa} and \eqref{AA1}. Let $\mathsf{R}_0 <1$. Let $(S_0,I_0)\in L_p^+(\Omega)\times L_1^+(J, L_p\Omega))$ and let $(S,I)$ be the corresponding positive global solution to~\eqref{E} provided by Theorem~\ref{T1}. Then
$$
\lim_{t\to\infty} (S(t),I(t))=(\kappa_2,0)\ \text{ in }\ L_p(\Omega)\times L_1(J, C(\bar\Omega))\,.
$$
\end{prop}

\begin{proof} Since solutions become immediately smooth according to Theorem~\ref{T1}, we may restrict without loss of generality to initial values
$$
S_0\in W_{p,N}^2(\Omega)\,,\quad I_0\in L_1(J, W_{p,N}^2(\Omega))\,,\qquad S_0, I_0\ge 0\,,\qquad S_0\not= 0\,.
$$

{\bf (i)} We first derive an upper bound on $S$. From \eqref{E3} we have
$$
\partial_t S(t,x)\le \Delta_N S(t,x)+\kappa_1\left(1-\dfrac{1}{\kappa_2}S(t,x)\right) S(t,x) 
$$
so that
\begin{equation}\label{Sbound}
S(t,x)\le z(t)\,,\quad (t,x)\in \R^+\times \Omega\,,
\end{equation}
by the comparison principle, where
$$
z(t):=\frac{\|S_0\|_\infty}{e^{-\kappa_1 t}\left(1-\frac{\kappa_1}{\kappa_2}\|S_0\|_\infty\right)+\frac{1}{\kappa_2}\|S_0\|_\infty}\,,\quad t\ge 0\,,
$$
is the solution to
$$
z'(t)=\kappa_1\big(1-\dfrac{1}{\kappa_2}z(t)\big) z(t)\,,\quad t\ge 0\,,\qquad z(0)=\|S_0\|_\infty\,.
$$
Due to $\mathsf{R}_0 <1$ and \eqref{specrad} we may choose $\ve_0>0$ such that
\begin{equation}\label{R11}
(\kappa_2+\ve_0)\int_0^{a_m} b(a)\, \Pi(a)\,\rd a <1\,.
\end{equation}
Since $\lim_{t\to\infty}z(t)=\kappa_2$, we find for fixed $\ve\in (0,\ve_0)$ some $t_0>0$ with 
\begin{align}\label{Se}
S(t,x)\le \kappa_2+\ve\,,\quad t\ge t_0\,,\quad x\in\Omega\,.
\end{align}

{\bf (ii)} Next, we derive an upper bound for $I$.  Using \eqref{Se} and \eqref{E1}-\eqref{E2} we obtain
\begin{align*}
D I(t+t_0,a)&=A(a)I(t+t_0,a)
\end{align*}
subject to
\begin{align*}
I(t+t_0,0)&\le (\kappa_2+\ve) \int_0^{a_m} b(a)\,I(t+t_0,a)\,\rd a\,, \quad t\ge 0\,.
\end{align*}
Let $G$ solve (see~\cite{WalkerIUMJ})
\begin{subequations}\label{p}
\begin{align} 
D G(t,a)&=A(a) G(t,a)\,, \quad t\ge 0\,,\quad \quad a\in J\,,
\\
G(t,0)&= (\kappa_2+\ve) \int_0^{a_m} b(a)\,G(t,a)\,\rd a\,, \qquad G(0,a)=I(t_0,a)\,.
\end{align}
\end{subequations}
We claim that  
\begin{align}\label{ww}
I(t+t_0,a)\le G(t,a)\,,\quad t\ge 0\,,\quad a\in J\,.
\end{align}  
Indeed, setting $w(t,a):=G(t,a)-I(t+t_0,a)$, we have, for $t\ge 0$ and $a\in J$,
\begin{align*}
D w(t,a)&=A(a) w(t,a) \,,  
\\
w(t,0)&\ge (\kappa_2+\ve) \int_0^{a_m} b(a)\,w(t,a)\,\rd a\,, \qquad w(0,a)=0\,,
\end{align*}
and therefore
\begin{align}\label{w}
w(t,a):=\left\{\begin{array}{ll}
0\,,& a>t\,, \ a\in J\,,\\[2pt]
U_A(a,0)\hat B(t-a) \,,&a\le t\,,\ a\in J\,.
\end{array} \right.
\end{align}
with
$$
\hat B(t)=w(t,0)\ge (\kappa_2+\ve) \int_0^{t} b(a)\,U_A(a,0)\,\hat B(t-a)\,\rd a\,,\quad t\ge 0\,.
$$
Introducing 
$$
(\mathcal{K}B)(t):= (\kappa_2+\ve) \int_0^{t} b(a)\,U_A(a,0)\,B(t-a)\,\rd a\,,\quad t\in [0,T]\,,\quad B\in C([0,T],L_p(\Omega))\,,
$$
it follows that $\mathcal{K}\in\ml\big(C([0,T],L_p(\Omega))\big)$ is a positive compact (Volterra) operator with spectral radius zero (see the proof of \cite[Lemma~5.1]{WalkerIUMJ}). Therefore, 
$$
(1-\mathcal{K}) ^{-1}=\sum_{k\ge 0}\mathcal{K}^k\ge 0
$$
and consequently
$$
\hat B=(1-\mathcal{K}) ^{-1} h\ge 0
$$
for
$$
h(t):=\hat B(t)- (\kappa_2+\ve) \int_0^{t} b(a)\,\hat B(t-a)\,\rd a\ge 0\,,\quad t\ge 0\,.
$$
Hence $w\ge 0$ according to \eqref{w} so that \eqref{ww} is true. \\

{\bf (iii)} Next, we claim that 
\begin{equation}\label{II}
I(t)\rightarrow 0 \ \text{ in }\ L_1(J,C(\bar\Omega))\ \text{ as }\ t\rightarrow \infty\,,
\end{equation}
which, according to \eqref{ww}, is ensured by showing that
\begin{align}\label{conv}
\lim_{t\to\infty} G(t)=0\ \text{ in }\ L_1(J, C(\bar\Omega))\,.
\end{align}
As for \eqref{conv} we fix $\alpha\in (n/2p,1)$ and  note that
problem~\eqref{p} with birth rate $(\kappa_2+\ve)b(a)$ fits exactly into the setting of problems investigated in \cite{WalkerIUMJ}. In fact, since $G(0)=I(t_0)\in L_1(J,W_{p,N}^{2\alpha}(\Omega))$, it follows from \cite[Corollary~1.3]{WalkerIUMJ} that $G(t)=e^{t\hat\A_\ve}G(0)$, $t\ge 0$, where $(e^{t\hat\A_\ve})_{t\ge 0}$ is an eventually compact, positive semigroup  on $L_1(J,W_{p,N}^{2\alpha}(\Omega))$. Therefore, \cite[V.~Corollary~3.2]{EngelNagel} ensures that the spectrum of the generator $\hat\A_\ve$ consists of eigenvalues only, while \cite[IV.~Corollary~3.12]{EngelNagel}  yields that the spectral bound of $\hat\A_\ve$ coincides with the type of the semigroup. In fact, the spectral bound  $s_0:=s(\hat\A_\ve)$ is the unique real number with $r(\hat Q_\ve^{s_0})=1$ according to \cite[Proposition 5.2]{WalkerIUMJ}, where
$$
\hat Q_\ve^\lambda:=(\kappa_2+\ve) \int_0^{a_m} b(a)\, U_A^\lambda(a,0)\,\rd a\,,\quad \lambda\in\C\,.
$$
Since as in~\eqref{specrad}
$$
r(\hat Q_\ve^{s_0})=(\kappa_2+\ve)\int_0^{a_m} b(a)\, \Pi(a) e^{-s_0a}\,\rd a \,,
$$
it follows from~\eqref{R11} that $s_0=s(\hat\A_\ve)<0$, and since the spectral bound and the type of the semigroup coincide, we conclude that
$$
\|G(t)\|_{L_1(J,W_{p,N}^{2\alpha}(\Omega))}\le N\, e^{-s_0 t}\,\|G(0)\|_{L_1(J,W_{p,N}^{2\alpha}(\Omega))}\longrightarrow 0
$$
as $t\to\infty$. Consequently, since $W_{p,N}^{2\alpha}(\Omega)\hookrightarrow C(\bar\Omega)$, we deduce~\eqref{conv} and therefore, owing to~\eqref{ww}, also~\eqref{II}.\\

{\bf (iv)} Finally, we prove that
$$
\lim_{t\to\infty} S(t)=\kappa_2 \ \text{ in }\ L_p(\Omega)\,.
$$
Given $\ve\in (0,\kappa_1)$ we infer from \eqref{II} that there is $t_1\ge t_0$ with
$$
\left\|\int_0^{a_m}b(a)I(t,a)\,\rd a\right\|_\infty\le \ve\,,\quad t\ge t_1\,,
$$
and hence, due to \eqref{E3},
\begin{align*}
\partial_t S(t,x)&\ge   \Delta_N S(t,x)+\kappa_1\left(1-\dfrac{1}{\kappa_2}S(t,x)\right) S(t,x)  -\ve S(t,x)\,,\quad t\ge t_1\,,\quad x\in\Omega\,.
\end{align*}
Since the strong maximum principle yields for every $x_0\in \Omega$ some $\varrho>0$ such that 
$$
\varrho_0:=\min_{\bar\B(x_0,\varrho)} S(t_1,\cdot)>0\,,
$$
 we obtain $S(t,x)\ge \xi(t)$ for $t\ge t_1$ and $x\in \bar\B(x_0,\varrho)$, where $\xi$ solves
$$
\xi'(t)=\kappa_1\left(1-\frac{\xi(t)}{\kappa_2}\right)\xi(t)-\ve \xi(t)\,,\quad t\ge t_1\,,\qquad \xi(t_1)=\xi_0\,.
$$ 
Therefore, 
$$
\liminf_{t\to\infty} S(t,x)\ge \lim_{t\to\infty}\xi(t)=\frac{\kappa_2(\kappa_1-\ve)}{\kappa_1}\,,\quad x\in \Omega\,.
$$
Letting $\ve\to 0$ and invoking \eqref{Se}, we derive
$$
\lim_{t\to\infty} S(t,x)=\kappa_2\,,\quad x\in \Omega\,.
$$
Finally, using again the $L_\infty$-bound from \eqref{Se} and Lebesgue's theorem we conclude that
$S(t)\to \kappa_2$ in~$L_p(\Omega)$ as $t\to\infty$.
Together with \eqref{II}, this  proves Proposition~\ref{P100}.
\end{proof}



\subsection{The Endemic Steady State $(\bar S_*,\bar I_*)$}

In case that the basic reproduction number satisfies
\begin{align*}
\mathsf{R}_0=\kappa_2\int_0^{a_m} b(a) \Pi(a)\, \rd a >1\,,
\end{align*}
there is an  endemic steady state $(\bar S_*,\bar I_*)$ given by
$$
\bar S_*:=\frac{\kappa_2}{\mathsf{R}_0 }\,,\qquad \bar I_*(a):= \frac{1}{\mathsf{R}_0 } \kappa_1\kappa_2\left(1-\frac{1}{\mathsf{R}_0 }\right)\Pi(a)\,,\quad a\in J\,.
$$ 
It is convenient to set $\mathsf{r}_0 :=1/\mathsf{R}_0 \in (0,1)$. Then, the endemic steady state can be written as
$$
\bar S_*=\mathsf{r}_0 \kappa_2\,,\qquad 
\bar I_*(a)=\Pi(a)i_*\,,\quad a\in J\,,\qquad i_*:=\mathsf{r}_0 \kappa_1\kappa_2\left(1-\mathsf{r}_0 \right)\,.
$$
In \eqref{STAR} we have
$$
q_*=\int_0^{a_m} b(a) \bar I_*(a)\,\rd a=\kappa_1\left(1-\mathsf{r}_0 \right)
$$
so that
$$
A_1^*=\Delta_N-\kappa_1\mathsf{r}_0 
$$
and
$$ 
P_*I=\mathsf{r}_0 \kappa_2 \int_0^{a_m}b(a)I(a)\,\rd a\,.
$$
In the following we still assume $p>(2\vee n)$. 

\begin{prop}\label{Prop51}
Assume ~\eqref{Aa} and \eqref{AA1}. For $1<\mathsf{R}_0 <3$, the endemic steady state  $(\bar S_*,\bar I_*)$ to~\eqref{E} is linearly stable in $L_p(\Omega)\times L_1(J,L_p(\Omega))$.
\end{prop}

\begin{proof}
Let $\lambda$ be a spectral point of the linearization, so that, according to Remark~\ref{R2}, there is a  nontrivial $(S,I)\in W_{p,N}^{2}(\Omega)\times C(J,L_p(\Omega))$ with
\begin{align*} 
\left(\begin{matrix} A_1^*  & -P_* \\ 0& -\partial_a+A(a) \end{matrix}\right)\left(\begin{matrix} S \\ I \end{matrix}\right) =\lambda \left(\begin{matrix} S \\ I \end{matrix}\right)\,,\qquad I(0)=P_*I +q_*S\,.
\end{align*}
That is,
\begin{subequations}
\begin{align}
-\Delta_NS&=-(\lambda+\kappa_1\mathsf{r}_0 ) S-\mathsf{r}_0 \kappa_2\int_0^{a_m}b(a)I(a)\,\rd a\,,  \label{L1y} \\ 
\partial_aI(a)&=\big(-\lambda+A(a)\big)I(a) \,,\qquad a\in J\,,\label{L2yx}\\ I(0)&=\mathsf{r}_0 \kappa_2\int_0^{a_m}b(a)I(a)\,\rd a+\kappa_1\left(1-\mathsf{r}_0 \right)S\,. \label{L2y}
\end{align}
\end{subequations}
From \eqref{L2yx} we get
$$
I(a)=U_A^\lambda(a,0)I(0)\,,\quad a\in J\,,
$$
and plugged into \eqref{L2y} this yields
\begin{align}\label{I1}
I(0)=\bar S_*Q^\lambda I(0) +\kappa_1\left(1-\mathsf{r}_0 \right)S 
\end{align}
with 
\begin{align}\label{QQQQq}
\bar S_*Q^\lambda=\mathsf{r}_0 \kappa_2\int_0^{a_m} b(a) U_A^\lambda(a,0)\,\rd a\in \ml\big(L_p(\Omega),W_{p,N}^{1}(\Omega)\big)\,.
\end{align}
In order to verify that $\mathrm{Re}\,\lambda<0$, we assume for contradiction that $\mathrm{Re}\,\lambda\ge 0$. Then  $\lambda+\kappa_1\mathsf{r}_0 -\Delta_N$ is invertible and we infer from~\eqref{L1y} and~\eqref{I1} that
 \begin{align}\label{I11}
(1-\bar S_*Q^\lambda) I(0) =-\kappa_1\left(1-\mathsf{r}_0 \right)\big(\lambda+\kappa_1\mathsf{r}_0 -\Delta_N\big)^{-1}S_*Q^\lambda I(0)\,.
\end{align}
Recall that the eigenfunctions $(\phi_j)_{j\in\N}$ of the Neumann-Laplacian, corresponding to the eigenvalues counted according to multiplicity,
 build an orthonormal basis in $W_{2,N}^1(\Omega)$. Then $-\Delta_N\phi_j=\mu_j \phi$ entails $e^{t\Delta_N}\phi_j=e^{-t\mu_j}\phi_j$ for $t\ge 0$, and the operator~$\bar S_*Q^\lambda$ leaves the eigenfunctions invariant. More precisely, from~\eqref{QQQ} we deduce
$$
\bar S_*Q^\lambda\phi_j=\mathcal{R}_{\lambda,\mu_j}\phi_j\,,\qquad \mathcal{R}_{\lambda,\mu_j}:=\mathsf{r}_0 \kappa_2\int_0^{a_m} b(a)\,\Pi(a)\, e^{-\lambda a}\,\exp\left(-\mu_j\int_0^ad(\sigma)\,\rd \sigma\right)\,\rd a\,,
$$
for every $j\in \N$. Note that $I(0)\in W_{p,N}^1(\Omega)\hookrightarrow W_{2,N}^1(\Omega)$ is nonzero as otherwise also $S=0$. Hence, writing $I(0)=\sum_j \xi_j\phi_j$, we derive from the identity~\eqref{I11} that
$$
\left(1-\mathcal{R}_{\lambda,\mu_j}\right)\xi_j=-\frac{\kappa_1\left(1-\mathsf{r}_0 \right)}{\lambda+\kappa_1\mathsf{r}_0 +\mu_j} \mathcal{R}_{\lambda,\mu_j}\xi_j\,,\quad j\in \N\,.
$$
Taking any $j\in\N$ with $\xi_j\not=0$, the previous identity leads to the characteristic equation
 \begin{align}\label{characteristicequation}
\frac{1}{\mathcal{R}_{\lambda,\mu_j}}= 1-\frac{1-\mathsf{r}_0 }{\frac{\lambda+\mu_j}{\kappa_1}+\mathsf{r}_0 }\,.
\end{align}
Owing to $\mu_j\ge 0$  we have
 \begin{align}\label{characteristicequation2}
\vert \mathcal{R}_{\lambda,\mu_j}\vert\le \mathcal{R}_{\mathrm{Re}\,\lambda,0}\le \mathcal{R}_{0,0}=\mathsf{r}_0  \mathsf{R}_0 =1\,,\qquad \mathrm{Re}\,\lambda\ge 0\,.
\end{align}
Clearly,  \eqref{characteristicequation} has no real solution $\lambda\ge 0$ since in this case 
$$
0<\mathcal{R}_{\lambda,\mu_j}\le 1\,,\qquad \frac{1-\mathsf{r}_0 }{\frac{\lambda+\mu_j}{\kappa_1}+\mathsf{r}_0 }>0\,.
$$ 
For an arbitrary $\lambda\in\C$ with $\mathrm{Re}\,\lambda\ge 0$ we write
$$
\zeta:=\frac{\lambda+\mu_j}{\kappa_1}=\alpha+i\beta\,,\qquad \alpha\ge 0\,,\quad \beta\in\R\,,
$$
and obtain from \eqref{characteristicequation} and  \eqref{characteristicequation2} the contradiction
\begin{align*}
1&\le \frac{1}{\vert \mathcal{R}_{\lambda,\mu_j}\vert^2}= \left\vert 1-\frac{1-\mathsf{r}_0 }{\zeta +\mathsf{r}_0 }\right\vert^2
=\frac{\vert \zeta +2\mathsf{r}_0 -1\vert^2}{\vert\zeta +\mathsf{r}_0 \vert^2}=1+\frac{(\mathsf{r}_0 -1)(2\alpha+3\mathsf{r}_0 -1)}{(\alpha+\mathsf{r}_0 )^2+\beta^2}<1
\end{align*}
since $\mathsf{r}_0 -1<0$ and $2\alpha+3\mathsf{r}_0 -1\ge 3\mathsf{r}_0 -1>0$ by our assumption that $1/3<\mathsf{r}_0 <1$. Therefore, we conclude that indeed $\mathrm{Re}\,\lambda<0$. Consequently, the endemic steady state  $(\bar S_*,\bar I_*)$ is linearly stable when $1<\mathsf{R}_0 <3$.
\end{proof}

Clearly, one expects $(\bar S_*,\bar I_*)$ to be linearly stable for all $\mathsf{R}_0 >1$.

\section{APPENDIX} \label{Appendix}

\subsection*{Regularity of $I$} We provide the missing step from the proof of Proposition~\ref{P2}.

\begin{lem}\label{APL1}
 For $\ve>0$ small and some  $n/2p<2\theta<2-n/p$, let
$$
S_\ve\in  C\big([0, T_m-\ve),W_{p,\mathcal{B}}^{2}(\Omega)\big)\,,\quad I_\ve\in  C\big([0, T_m-\ve),L_1(J,W_{p,\mathcal{B}}^{2\theta}(\Omega))\big)\,,\quad  I_{0,\ve}\in L_1\big(J,W_{p,\mathcal{B}}^{2\theta}(\Omega)\big)
$$
be as in the proof of Proposition~\ref{P2}. Then 
$$I_\ve\in C\big((0,T_m-\ve),L_1(J,W_{p,\mathcal{B}}^{2}(\Omega))\big)\,,\qquad 
I\in C\big((0,T_m),L_1(J,W_{p,\mathcal{B}}^{2}(\Omega))\big)\,.
$$
\end{lem}

\begin{proof}
We proceed analogously to the proof of Proposition~\ref{P1}.
According to Lemma~\ref{L1w} there is $\alpha>0$ such that
$$
B[S_\ve,I_\ve] \in C\big([0, T_m-\ve),W_{p,\mathcal{B}}^{2\alpha}(\Omega)\big)
$$
and we obtain from \eqref{EST}, for $0<t_2\le t_1\le T<T_m-\ve$,
\begin{align*}
||I_\ve(t_1,\cdot)&-I_\ve(t_2,\cdot)\|_{L_1(J,W_{p,\mathcal{B}}^{2}(\Omega))}\\
&\le \int_0^{t_2}\|U_A(a,0)\|_{\ml(W_{p,\mathcal{B}}^{2\alpha}(\Omega),W_{p,\mathcal{B}}^{2}(\Omega))}\,\|B[S_\ve,I_\ve](t_1-a)-B[S_\ve,I_\ve](t_2-a)\|_{W_{p,\mathcal{B}}^{2\alpha}(\Omega)}\,\rd a\\
&\quad +\int_{t_2}^{t_1}\|U_A(a,0)\|_{\ml(W_{p,\mathcal{B}}^{2\alpha}(\Omega),W_{p,\mathcal{B}}^{2}(\Omega))}\,\|B[S_\ve,I_\ve](t_1-a)\|_{W_{p,\mathcal{B}}^{2\alpha}(\Omega)}\,\rd a\\
&\quad + \int_{t_2}^{t_1}\|U_A(a,a-t_2)\|_{\ml(W_{p,\mathcal{B}}^{2\theta}(\Omega),W_{p,\mathcal{B}}^{2}(\Omega))}\,\|I_{0,\ve}(a-t_2)\|_{W_{p,\mathcal{B}}^{2\theta}(\Omega)}\,\rd a\\
&\quad + \int_{t_1}^{a_m}\big\| \big(U_A(a,a-t_1)-U_A(a,a-t_2)\big) I_{0,\ve}(a-t_1)\|_{W_{p,\mathcal{B}}^{2}(\Omega)}\,\rd a\\
&\quad + \int_{t_1}^{a_m} \| U_A(a,a-t_2)\|_{\ml(W_{p,\mathcal{B}}^{2\theta}(\Omega),W_{p,\mathcal{B}}^{2}(\Omega))}\,\| I_{0,\ve}(a-t_1)-I_{0,\ve}(a-t_2)\|_{W_{p,\mathcal{B}}^{2\theta}(\Omega)}\,\rd a\\
&\le Me^\varpi \int_0^{t_2} a^{\alpha-1}\,\|B[S_\ve,I_\ve](t_1-a)-B[S_\ve,I_\ve](t_2-a)\|_{W_{p,\mathcal{B}}^{2\alpha}(\Omega)}\,\rd a\\
&\quad +c(R) \int_{t_2}^{t_1}a^{\alpha-1}\,\rd a + t_2^{\theta-1}\int_{t_2}^{t_1}\|I_{0,\ve}(a-t_2)\|_{W_{p,\mathcal{B}}^{2\theta}(\Omega)}\,\rd a\\
&\quad + \int_{t_1}^{a_m}\big\| \big(U_A(a,a-t_1)-U_A(a,a-t_2)\big) I_{0,\ve}(a-t_1)\|_{W_{p,\mathcal{B}}^{2}(\Omega)}\,\rd a\\
&\quad + t_2^{\theta-1} \int_{t_1}^{a_m} \| I_{0,\ve}(a-t_1)-I_{0,\ve}(a-t_2)\|_{W_{p,\mathcal{B}}^{2\theta}(\Omega)}\,\rd a\,.
\end{align*}
Now, as $\vert t_1-t_2\vert \to 0$, the first integral on the right-hand side goes to zero since the function $B[S_\ve,I_\ve]\in C([0,T],W_{p,\mathcal{B}}^{2\alpha}(\Omega))$ is uniformly continuous while the second and the third integral vanish since $a\mapsto a^{\alpha-1}$ respectively $I_{0,\ve}$ are integrable. To see that the fourth integral vanishes in the limit one may use the strong continuity \cite[Equation~II.~(2.1.2)]{LQPP} of the evolution operator $U_A$ in $\ml\big(W_{p,\mathcal{B}}^{2\theta}(\Omega),W_{p,\mathcal{B}}^{2}(\Omega)\big)$  and Lebesgue's theorem. For the last integral one may use the strong continuity of the translations on $L_1\big(J,W_{p,\mathcal{B}}^{2\theta}(\Omega)\big)$.
Consequently, $I_\ve\in C\big((0,T_m-\ve),L_1(J,W_{p,\mathcal{B}}^{2}(\Omega))\big)$. Letting $\ve\to 0$ yields
$I\in C\big((0,T_m),L_1(J,W_{p,\mathcal{B}}^{2}(\Omega))\big)$.
\end{proof}

\subsection*{Proof of the $L_1$-Inequality~\eqref{L1id}}\label{APSec2}

We derive inequality~\eqref{L1id}. To this end, note from Gauss' theorem that
\begin{equation}\label{AB}
\int_\Omega \Delta u\,\rd x=\int_{\partial\Omega} \partial_\nu u\,\rd \sigma \le 0\,,\qquad u\in W_{p,\mathcal{B}}^2(\Omega)\,,\quad u\ge 0\,, 
\end{equation}
since $\partial_\nu u=0$ on $\partial\Omega$ if $\delta=1$ and $\partial_\nu u\le 0$ on $\partial\Omega$ if $\delta=0$.
Now, for $a\in J$ fixed set
$$
w(t):=U_A(a+t,a)I_0(a)\,,\quad t\in [0,a_m-a]\,.
$$
Integrating then 
$$
\frac{\rd}{\rd t} w(t)=A(a+t)w(t)\,,\quad t\in (0,a_m-a]\,,
$$
with $A$ given in \eqref{A}, we get from \eqref{AB}
\begin{align*}
\int_\Omega w(t,x)\,\rd x\le \int_\Omega w(0,x)\,\rd x-\int_0^t\int_\Omega \big(m(a+\tau,x)+r(a+\tau,x)\big)\, w(\tau,x)\,\rd x\,\rd \tau\,,
\end{align*}
and therefore
\begin{align*}
\int_\Omega U_A(a+t,a)I_0(a)\,\rd x\le \int_\Omega I_0(a)\,\rd x
-\int_0^t\int_\Omega \big(m(a+\tau)+r(a+\tau)\big)\, U_A(a+\tau,a)I_0(a)\,\rd x\,\rd \tau\,.
\end{align*}
Similarly, one derives for $t>a$ that
\begin{align*}
\int_\Omega U_A(&t-a,0)I(a,0)\,\rd x\\
&\le \int_\Omega I(a,0)\,\rd x-\int_a^t\int_\Omega \big(m(\tau-a)+r(\tau-a)\big)\, U_A(\tau-a,0)I(a,0)\,\rd x\,\rd \tau\,.
\end{align*}
We then recall~\eqref{I} and use the previous two identities to obtain
\begin{align*}
\int_0^{a_m}\int_\Omega I(t,a)\,\rd x\,\rd a&= \int_0^t\int_\Omega U_A(t-a,0) I(a,0)\,\rd x\rd a+\int_0^{a_m-t}\int_\Omega U_A(a+t,a)I_0(a)\,\rd x\,\rd a\\
&\le\int_0^t\int_\Omega I(a,0)\,\rd x\,\rd a +\int_0^{a_m-t}\int_\Omega I_0(a)\,\rd x\,\rd a\\
&\quad -
\int_0^t\int_a^t\int_\Omega \big(m(\tau-a)+r(\tau-a)\big)\, U_A(\tau-a,0)I(a,0)\,\rd x\,\rd \tau\,\rd a\\
&\quad -\int_0^{a_m-t}\int_0^t\int_\Omega \big(m(a+\tau)+r(a+\tau)\big)\, U_A(a+\tau,a)I_0(a)\,\rd x\,\rd \tau\,\rd a\\
&=\int_0^t\int_\Omega I(a,0)\,\rd x\,\rd a +\int_0^{a_m}\int_\Omega I_0(a)\,\rd x\,\rd a\\
&\quad -
\int_0^t\int_0^\tau\int_\Omega \big(m(a)+r(a)\big)\, U_A(a,0)I(\tau-a,0)\,\rd x\,\rd a\,\rd  \tau\\
&\quad -\int_0^t\int_\tau^{a_m}\int_\Omega \big(m(a)+r(a)\big)\, U_A(a,a-\tau)I_0(a-\tau)\,\rd x\,\rd a\,\rd \tau\\
&\quad - \int_{a_m-t}^{a_m}\int_\Omega I_0(a)\,\rd x\,\rd a\\
&\quad + \int_0^t\int_{a_m-t+\tau}^{a_m}\int_\Omega \big(m(a)+r(a)\big)\, U_A(a,a-\tau)I_0(a-\tau)\,\rd x\,\rd a\,\rd \tau\,.
\end{align*}
Using~\eqref{I} again we may rewrite the previous inequality as
\begin{equation}\begin{split}\label{P11}
\int_0^{a_m}\int_\Omega I(t,a)\,\rd x\,\rd a
&\le\int_0^{a_m}\int_\Omega I_0(a)\,\rd x\,\rd a +\int_0^t\int_\Omega I(\tau,0)\,\rd x\,\rd \tau \\
&\quad -
\int_0^t\int_0^{a_m}\int_\Omega \big(m(a)+r(a)\big)\,I(\tau,a)\,\rd x\,\rd a\,\rd \tau\\
&\quad - \int_{a_m-t}^{a_m}\int_\Omega I_0(a)\,\rd x\,\rd a\\
&\quad + \int_0^t\int_{a_m-t+\tau}^{a_m}\int_\Omega \big(m(a)+r(a)\big)\, U_A(a,a-\tau)I_0(a-\tau)\,\rd x\,\rd a\,\rd \tau\,.
\end{split}\end{equation}
Integrating~\eqref{E2} and using \eqref{AB} yields
\begin{equation}\begin{split}\label{P22}
\int_\Omega S(t)\,\rd x\le &\int_\Omega S_0\,\rd x+\int_0^{a_m}\int_\Omega I_0(a)\,\rd x\,\rd a+\int_0^t\int_\Omega \kappa_1\left(1-\frac{S(\tau)}{\kappa_2}\right)S(\tau)\,\rd x\,\rd \tau\\
& -\int_0^t\int_\Omega I(\tau,0)\,\rd x\,\rd \tau+\int_0^t\int_\Omega\int_0^{a_m} r(a) I(\tau,a)\,\rd a\,\rd x\,\rd\tau\,.
\end{split}\end{equation}
Adding~\eqref{P11} and~\eqref{P22} yields~\eqref{L1id}.\qed

\begin{rem}
When considering Neumann boundary conditions ($\delta=1$), then~\eqref{P11} and~\eqref{P22} are equalities and thus also~\eqref{L1id}. 
\end{rem}

\section*{Acknowledgment}
I thank Lina Sophie Schmitz for helpful discussions on the topic.



\bibliographystyle{siam}
\bibliography{AgeDiff_220203}

\begin{thebibliography}{10}

\bibitem{AmannSIAM}
{\sc H.~Amann}, {\em Fixed point equations and nonlinear eigenvalue problems in
  ordered {B}anach spaces}, SIAM Rev., 18 (1976), pp.~620--709.

\bibitem{AmannIsrael}
\leavevmode\vrule height 2pt depth -1.6pt width 23pt, {\em Dual semigroups and
  second order linear elliptic boundary value problems}, Israel J. Math., 45
  (1983), pp.~225--254.

\bibitem{AmannMultiplication}
\leavevmode\vrule height 2pt depth -1.6pt width 23pt, {\em Multiplication in
  {S}obolev and {B}esov spaces}, in Nonlinear analysis, Sc. Norm. Super. di
  Pisa Quaderni, Scuola Norm. Sup., Pisa, 1991, pp.~27--50.

\bibitem{LQPP}
\leavevmode\vrule height 2pt depth -1.6pt width 23pt, {\em Linear and
  quasilinear parabolic problems. {V}ol. {I}}, vol.~89 of Monographs in
  Mathematics, Birkh\"{a}user Boston, Inc., Boston, MA, 1995.
\newblock Abstract linear theory.

\bibitem{AmannMaxPrinc}
{\sc H.~Amann}, {\em Maximum principles and principal eigenvalues}, in Ten
  mathematical essays on approximation in analysis and topology, Elsevier B.
  V., Amsterdam, 2005, pp.~1--60.

\bibitem{BarasHassanVeron}
{\sc P.~Baras, J.-C. Hassan, and L.~V\'{e}ron}, {\em Compacit\'{e} de
  l'op\'{e}rateur d\'{e}finissant la solution d'une \'{e}quation
  d'\'{e}volution non homog\`ene}, C. R. Acad. Sci. Paris S\'{e}r. A-B, 284
  (1977), pp.~A799--A802.

\bibitem{BlatBrown}
{\sc J.~Blat and K.~J. Brown}, {\em Global bifurcation of positive solutions in
  some systems of elliptic equations}, SIAM J. Math. Anal., 17 (1986),
  pp.~1339--1353.

\bibitem{CaoYanXu}
{\sc H.~Cao, D.~Yan, and X.~Xu}, {\em Hopf bifurcation for an {SIR} model with
  age structure}, Math. Model. Nat. Phenom., 16 (2021), pp.~Paper No. 7, 17.

\bibitem{ChekrounKuniya19}
{\sc A.~Chekroun and T.~Kuniya}, {\em An infection age-space-structured {SIR}
  epidemic model with {D}irichlet boundary condition}, Math. Model. Nat.
  Phenom., 14 (2019), pp.~Paper No. 505, 22.

\bibitem{ChekrounKuniya20a}
\leavevmode\vrule height 2pt depth -1.6pt width 23pt, {\em Global threshold
  dynamics of an infection age-structured {SIR} epidemic model with diffusion
  under the {D}irichlet boundary condition}, J. Differential Equations, 269
  (2020), pp.~117--148.

\bibitem{ChekrounKuniya20b}
\leavevmode\vrule height 2pt depth -1.6pt width 23pt, {\em An infection
  age-space structured {SIR} epidemic model with {N}eumann boundary condition},
  Appl. Anal., 99 (2020), pp.~1972--1985.

\bibitem{DanersKochMedina}
{\sc D.~Daners and P.~Koch~Medina}, {\em Abstract evolution equations, periodic
  problems and applications}, vol.~279 of Pitman Research Notes in Mathematics
  Series, Longman Scientific \& Technical, Harlow; copublished in the United
  States with John Wiley \& Sons, Inc., New York, 1992.

\bibitem{DiBlasio10}
{\sc G.~Di~Blasio}, {\em Mathematical analysis for an epidemic model with
  spatial and age structure}, J. Evol. Equ., 10 (2010), pp.~929--953.

\bibitem{DucrotMagal09}
{\sc A.~Ducrot and P.~Magal}, {\em Travelling wave solutions for an
  infection-age structured model with diffusion}, Proc. Roy. Soc. Edinburgh
  Sect. A, 139 (2009), pp.~459--482.

\bibitem{DucrotMagal11}
{\sc A.~Ducrot and P.~Magal}, {\em Travelling wave solutions for an
  infection-age structured epidemic model with external supplies},
  Nonlinearity, 24 (2011), pp.~2891--2911.

\bibitem{DucrotMagal10}
{\sc A.~Ducrot, P.~Magal, and S.~Ruan}, {\em Travelling wave solutions in
  multigroup age-structured epidemic models}, Arch. Ration. Mech. Anal., 195
  (2010), pp.~311--331.

\bibitem{EngelNagel}
{\sc K.-J. Engel and R.~Nagel}, {\em One-parameter semigroups for linear
  evolution equations}, vol.~194 of Graduate Texts in Mathematics,
  Springer-Verlag, New York, 2000.
\newblock With contributions by S. Brendle, M. Campiti, T. Hahn, G. Metafune,
  G. Nickel, D. Pallara, C. Perazzoli, A. Rhandi, S. Romanelli and R.
  Schnaubelt.

\bibitem{FitzgibbonParrotWebb95}
{\sc W.~E. Fitzgibbon, M.~E. Parrott, and G.~F. Webb}, {\em Diffusive epidemic
  models with spatial and age dependent heterogeneity}, Discrete Contin. Dynam.
  Systems, 1 (1995), pp.~35--57.

\bibitem{FitzgibbonParrotWebb96}
\leavevmode\vrule height 2pt depth -1.6pt width 23pt, {\em A diffusive
  age-structured {SEIRS} epidemic model}, Methods Appl. Anal., 3 (1996),
  pp.~358--369.

\bibitem{KangRuanJMB21}
{\sc H.~Kang and S.~Ruan}, {\em Mathematical analysis on an age-structured
  {SIS} epidemic model with nonlocal diffusion}, J. Math. Biol., 83 (2021),
  p.~5.

\bibitem{Kim06}
{\sc M.-Y. Kim}, {\em Global dynamics of approximate solutions to an
  age-structured epidemic model with diffusion}, Adv. Comput. Math., 25 (2006),
  pp.~451--474.

\bibitem{KuboLanglais94}
{\sc M.~Kubo and M.~Langlais}, {\em Periodic solutions for nonlinear population
  dynamics models with age-dependence and spatial structure}, J. Differential
  Equations, 109 (1994), pp.~274--294.

\bibitem{KuniyaOizumi15}
{\sc T.~Kuniya and R.~Oizumi}, {\em Existence result for an age-structured
  {SIS} epidemic model with spatial diffusion}, Nonlinear Anal. Real World
  Appl., 23 (2015), pp.~196--208.

\bibitem{LanglaisBusenberg97}
{\sc M.~Langlais and S.~Busenberg}, {\em Global behaviour in age structured
  {S}.{I}.{S}. models with seasonal periodicities and vertical transmission},
  J. Math. Anal. Appl., 213 (1997), pp.~511--533.

\bibitem{Rothe}
{\sc F.~Rothe}, {\em Global solutions of reaction-diffusion systems}, vol.~1072
  of Lecture Notes in Mathematics, Springer-Verlag, Berlin, 1984.

\bibitem{ThiemeBook}
{\sc H.~R. Thieme}, {\em Mathematics in population biology}, Princeton Series
  in Theoretical and Computational Biology, Princeton University Press,
  Princeton, NJ, 2003.

\bibitem{Triebel}
{\sc H.~Triebel}, {\em Interpolation theory, function spaces, differential
  operators}, vol.~18 of North-Holland Mathematical Library, North-Holland
  Publishing Co., Amsterdam-New York, 1978.

\bibitem{WalkerDCDSA10}
{\sc {\relax Ch}.~Walker}, {\em Age-dependent equations with non-linear
  diffusion}, Discrete Contin. Dyn. Syst., 26 (2010), pp.~691--712.

\bibitem{WalkerCrelle}
\leavevmode\vrule height 2pt depth -1.6pt width 23pt, {\em On positive
  solutions of some system of reaction-diffusion equations with nonlocal
  initial conditions}, J. Reine Angew. Math., 660 (2011), pp.~149--179.

\bibitem{WalkerMOFM}
\leavevmode\vrule height 2pt depth -1.6pt width 23pt, {\em Some remarks on the
  asymptotic behavior of the semigroup associated with age-structured diffusive
  populations}, Monatsh. Math., 170 (2013), pp.~481--501.

\bibitem{WalkerIUMJ}
\leavevmode\vrule height 2pt depth -1.6pt width 23pt, {\em Properties of the
  semigroup in {$L$}$_1$~associated with age-structured diffusive populations}.
\newblock To appear in Indiana Univ. Math. J. (arXiv: 2109.01573), (2021).

\bibitem{WalkerInstabiliy}
\leavevmode\vrule height 2pt depth -1.6pt width 23pt, {\em Stability and
  instability of equilibria in age-structured diffusive populations}.
\newblock Preprint, arXiv:2304.09589 ({\tt https://arxiv.org/pdf/2304.09589}),
  (2023).

\bibitem{WalkerZehetbauer}
{\sc {\relax Ch}.~Walker and J.~Zehetbauer}, {\em The principle of linearized
  stability in age-structured diffusive populations}, J. Differential
  Equations, 341 (2022), pp.~620--656.

\bibitem{WebbARMA80}
{\sc G.~F. Webb}, {\em An age-dependent epidemic model with spatial diffusion},
  Arch. Rational Mech. Anal., 75 (1980/81), pp.~91--102.

\bibitem{WebbJMB82}
\leavevmode\vrule height 2pt depth -1.6pt width 23pt, {\em A recovery-relapse
  epidemic model with spatial diffusion}, J. Math. Biol., 14 (1982),
  pp.~177--194.

\bibitem{WebbBook}
\leavevmode\vrule height 2pt depth -1.6pt width 23pt, {\em Theory of nonlinear
  age-dependent population dynamics}, vol.~89 of Monographs and Textbooks in
  Pure and Applied Mathematics, Marcel Dekker, Inc., New York, 1985.

\bibitem{WebbSpringer}
\leavevmode\vrule height 2pt depth -1.6pt width 23pt, {\em Population models
  structured by age, size, and spatial position}, in Structured population
  models in biology and epidemiology, vol.~1936 of Lecture Notes in Math.,
  Springer, Berlin, 2008, pp.~1--49.

\end{thebibliography}

\end{document}